\documentclass[12pt]{amsart}
\allowdisplaybreaks
\usepackage[english]{babel}
\usepackage[pdftex,paper=a4paper,portrait=true,textwidth=450pt,textheight=675pt,tmargin=3cm,marginratio=1:1]{geometry}
\usepackage{amsfonts}
\usepackage[dvips]{graphics}
\usepackage[colorinlistoftodos]{todonotes}
\usepackage{amsmath}
\usepackage{amsthm}
\usepackage{amssymb}
\usepackage{bbm}
\usepackage{cancel}
\usepackage{color}
\usepackage[curve]{xypic}
\usepackage{graphicx}
\newtheorem{theorem}{Theorem}[section]

\newtheorem{lemma}[theorem]{Lemma}
\newtheorem{conjecture}[theorem]{Conjecture}

\theoremstyle{definition}

\newtheorem{remark}[theorem]{Remark}

\theoremstyle{property}

\DeclareFontFamily{OT1}{rsfs}{}
\DeclareFontShape{OT1}{rsfs}{n}{it}{<-> rsfs10}{}
\DeclareMathAlphabet{\curly}{OT1}{rsfs}{n}{it}

\newcommand\I{\mathcal I}

\renewcommand\O{\mathcal O}
\newcommand\PP{\mathbb P}

\newcommand\cE{\mathcal E}
\newcommand\EE{\mathbb E}

\newcommand\F{\mathcal F}

\newcommand\sfL{\mathsf L}

\newcommand\E{\mathbb E}
\newcommand\bfv{{\mathbf v}}
\newcommand\C{\mathbb C}
\newcommand\cC{\mathcal C}

\newcommand\FF{\mathbb F}

\newcommand\sfZ{\mathsf Z}

\newcommand\Q{\mathbb Q}

\newcommand\R{\mathbb R}

\newcommand\Z{\mathbb Z}
\newcommand\cZ{\mathcal Z}

\newcommand\Res{\mathrm{Res}}
\newcommand\inst{\mathrm{inst}}
\newcommand\mono{\mathrm{mono}}

\newcommand\bsa{\boldsymbol{a}}
\newcommand\bsn{\boldsymbol{n}}

\renewcommand\t{\mathfrak t}

\newcommand\SU{\mathrm{SU}}

\newcommand\vd{\mathrm{vd}}
\newcommand\pt{\mathrm{pt}}
\newcommand\vir{\mathrm{vir}}

\newcommand\cob{\operatorname{cob}}
\newcommand\elg{\operatorname{ell}}

\newcommand\SW{\mathrm{SW}}
\newcommand\VW{\mathrm{VW}}
\newcommand\td{\mathrm{td}}
\newcommand\frakT{\mathfrak{T}}
\newcommand\rk{\operatorname{rk}}

\newcommand\tr{\operatorname{tr}}

\newcommand\ch{\operatorname{ch}}
\newcommand\ev{\operatorname{ev}}

\newcommand\Hom{\operatorname{Hom}}
\renewcommand\hom{\mathcal{H}{\it{om}}}
\newcommand\Ext{\operatorname{Ext}}

\newcommand\Pic{\operatorname{Pic}}

\newcommand\Spec{\operatorname{Spec}\,}

\newcommand\Sym{\operatorname{Sym}}

\newcommand\INTO{\ar@{^{(}->}[r]}

\newcommand{\smfr}[2]{\hbox{$\frac{#1}{#2}$}}

\DeclareRobustCommand{\SkipTocEntry}[4]{}

\begin{document}
\title[Sheaves on surfaces and virtual invariants]{Sheaves on surfaces and virtual invariants}
\author[G\"ottsche and Kool]{L.~G\"ottsche and M.~Kool}
\maketitle
\centerline{\emph{Dedicated to Prof.~S.-T.~Yau, on the occasion of his 70th birthday.}}

\begin{abstract}
Moduli spaces of stable sheaves on smooth projective surfaces are in general singular. Nonetheless, they carry a virtual class, which ---in analogy with the classical case of Hilbert schemes of points--- can be used to define intersection numbers, such as virtual Euler characteristics, Verlinde numbers, and Segre numbers. 

We survey a set of recent conjectures by the authors for these numbers with applications to Vafa-Witten theory, $K$-theoretic S-duality, a rank 2 Dijkgraaf-Moore-Verlinde-Verlinde formula, and a virtual Segre-Verlinde correspondence. A key role is played by Mochizuki's formula for descendent Donaldson invariants.
\end{abstract}
\thispagestyle{empty}
\tableofcontents

\section{Introduction}

Hilbert schemes parametrizing closed subschemes of a quasi-projective variety were introduced by A.~Grothendieck \cite{Gro}. The case of 0-dimensional subschemes of an irreducible smooth projective surface $S$ has attracted a lot of attention. The Hilbert scheme $S^{[n]}$, parametrizing 0-dimensional subschemes $Z \subset S$ of length $n$, is irreducible and smooth of dimension $2n$ by a result of J.~Fogarty \cite{Fog}. Particularly notable are H.~Nakajima's operators on the direct sum (over all $n$) of the cohomology of $S^{[n]}$, which make
it into an irreducible representation of the Heisenberg algebra \cite{Nak1, Groj}. We will not survey the vast literature on Hilbert schemes of points on surfaces. Instead, we briefly discuss two invariants, namely their topological Euler characteristics and Segre numbers. \\

\noindent \textbf{Euler characteristics.} In 1990, the first-named author determined the Betti numbers of $S^{[n]}$ \cite{Got1}. The formula specializes to the following expression for the Euler characteristics $e(S^{[n]})$ of $S^{[n]}$ in terms of the Euler characteristic $e(S)$ of $S$
\begin{equation} \label{eulerrk1}
\sum_{n=0}^{\infty} e(S^{[n]}) \, q^n = \prod_{n=1}^{\infty} (1-q^n)^{-e(S)}.
\end{equation}
Up to a factor, this is equal to $\eta(q)^{-e(S)}$, where $\eta(q)$ denotes the Dedekind eta function
$$
\eta(q) = q^{\frac{1}{24} }\prod_{n=1}^{\infty} (1-q^n).
$$
The appearance of a function with ``modular properties'' is related to a symmetry in physics called S-duality \cite{VW}, which we discuss in detail in Sections \ref{sec:VW} and \ref{sec:rk3Sduality}. 

Formula \eqref{eulerrk1} has a beautiful application to enumerative geometry discovered by S.-T.~Yau and E.~Zaslow \cite{YZ}. Let $|L|$ be a ``general'' complete linear system on a K3 surface (containing only irreducible reduced curves, which are at worst nodal cf.~\cite{Che}). Then $|L|$ contains finitely many rational curves. Their number, $n_g$, only depends on the arithmetic genus $g$ of $|L|$ given by 
$
2g-2 = L^2.
$
The famous Yau-Zaslow formula states
$$
\sum_{g=0}^{\infty} n_g \, q^{g-1} = \Delta(q)^{-1},
$$
where $\Delta(q) = \eta(q)^{24}$ is the discriminant modular form. The idea of Yau-Zaslow is to realize $n_g$ as the Euler characteristic of the relative compactified Jacobian $\overline{\mathrm{Jac}}^g(\mathcal{C} / |L|)$ of degree $g$ line bundles on the fibres of the universal curve $\cC \rightarrow |L|$. Since $\overline{\mathrm{Jac}}^g(\mathcal{C} / |L|)$ is birational to $S^{[g]}$, and both are holomorphic symplectic, their Euler characteristics are equal. Formula \eqref{eulerrk1} for $e(S)=24$ yields the result. The influence of the Yau-Zaslow formula on enumerative geometry can be measured by the large number of essentially different proofs \cite{BL, Beau, KMPS, MPT, PT, Tod}. \\

\noindent \textbf{Segre numbers.} Let $L$ be a line bundle on a smooth projective surface $S$. Denote by $\cZ \subset S \times S^{[n]}$ the universal subscheme and consider the projections $p : \cZ \to S$ and $q : \cZ \to S^{[n]}$.
For any line bundle $L$ on $S$, one defines the corresponding tautological vector bundle by $L^{[n]} = q_* p^*L$. The Segre numbers are defined by
\begin{equation} \label{Segrenumber}
\int_{S^{[n]}} s_{2n}(L^{[n]}),
\end{equation}
where $s_{2n}$ denotes the degree $2n$ Segre class. In 1999, M.~Lehn \cite{Leh} conjectured the following remarkable formula 
\begin{equation} \label{lehnconj}
\sum_{n=0}^{\infty} \int_{S^{[n]}} s_{2n}(L^{[n]}) \, z^n = \frac{(1-t)^{L K_S - 2K_S^2} (1-2t)^{(L-K_S)^2+3\chi(\O_S)}}{(1-6t+6t^2)^{\frac{1}{2}L(L-K_S) + \chi(\O_S)}}, 
\end{equation}
where
$$
z = \frac{t(1-t)(1-2t)^4}{(1-6t+6t^2)^3}.
$$
Lehn's conjecture for $K_S$-trivial surfaces was proved by A.~Marian, D.~Oprea, and R.~Pandharipande \cite{MOP1}. The general case was established by the same authors in \cite{MOP2} building on \cite{MOP1} and work of C.~Voisin \cite{Voi}. 

The Segre number \eqref{Segrenumber} has an interesting interpretation in enumerative geometry. For $S \hookrightarrow \PP^{3n-2}$ and $L \cong \O(1)|_S$, \eqref{Segrenumber} counts the number of $(n-2)$-dimensional projective linear subspaces of $\PP^{3n-2}$ that are $n$-secant to $S$. \\

These Euler characteristics and Segre numbers are examples of intersection numbers on $S^{[n]}$. More precisely, both can be expressed in terms of polynomial expressions in Chern classes of tautological bundles $L^{[n]}$ and the holomorphic tangent bundle $T_{S^{[n]}}$. Indeed, by the Poincar\'e-Hopf index theorem
$$
e(S) = \int_{S^{[n]}} c_{2n}(T_{S^{[n]}}).
$$
Now let $P(L)$ be \emph{any} polynomial expression in terms of Chern classes of $L^{[n]}$ and $T_{S^{[n]}}$. Using nested Hilbert schemes, parametrizing $Z_0 \subset Z_1 \subset S$ with $\ell(Z_1 \setminus Z_0) = 1$, the first-named author, Lehn, and G.~Ellingsrud proved that there exists a polynomial $Q \in \Q[x,y,z,w]$, independent of $S$ and $L$, with the following universal property \cite[Thm.~4.1]{EGL}. For any line bundle $L$ on any smooth projective surface $S$, we have
$$
\int_{S^{[n]}} P(L) = Q(L^2, L K_S, K_S^2, \chi(\O_S)).
$$
This result is often the first step in proofs of identities like \eqref{eulerrk1} and \eqref{lehnconj}. For instance, together with the multiplicative nature of total Chern and Segre classes, \eqref{eulerrk1} and \eqref{lehnconj} are determined by two resp.~four universal series. More precisely, there exist $A,B \in 1+q \, \Q[[q]]$ and $W,X,Y,Z \in 1+z \, \Q[[z]]$ such that
\begin{align*}
\sum_{n=0}^{\infty} e(S^{[n]}) \, q^n &= A^{\chi(\O_S)} B^{K_S^2}, \\
\sum_{n=0}^{\infty} \int_{S^{[n]}} s_{2n}(L^{[n]}) \, z^n &= W^{L^2} X^{\chi(\O_S)} Y^{L K_S} Z^{K_S^2}.
\end{align*}
In the first case, $A, B$ are easily determined: evaluate on $S = \PP^2$ and $S = \PP^1 \times \PP^1$, which are toric surfaces with torus $T$. The torus action lifts to $S^{[n]}$ and $e(S^{[n]})$ equals the Euler characteristic of its fixed point locus $(S^{[n]})^T \subset S^{[n]}$. These fixed loci are described by collections of monomial ideals of total colength $n$, so the problem is reduced to Euler's formula for enumerating partitions. In contrast, proving Lehn's conjectural formulae for $W, X, Y, Z$ is much harder and took almost two decades \cite{MOP1, Voi, MOP2}.

\begin{remark}
As an aside, we briefly mention another application of the universality result of \cite{EGL} to enumerative geometry. The generalization of the Yau-Zaslow formula to counting nodal curves of any genus in any (sufficiently ample) complete linear system $|L|$ on any smooth projective surface $S$ is known as the G\"ottsche-Yau-Zaslow formula \cite{Got3}. There are now many proofs of this formula (in algebraic geometry by \cite{Tze, KST, Ren}, see the introduction to \cite{KST} for references to other approaches). In \cite{KST}, the G\"ottsche-Yau-Zaslow formula is expressed in terms of intersection numbers of the form
$$
\int_{S^{[n]} \times \PP} c_n(L^{[n]}) \frac{c(T_{S^{[n]} \times \PP}   )}{ c(L^{[n]} (1))},
$$
where $\PP \subset |L|$ is an appropriate linear subsystem, $c$ denotes total Chern class, and $\O(1)$ is the tautological line bundle on $\PP$. Universality implies that the G\"ottsche-Yau-Zaslow formula is determined by four universal functions.
\end{remark}

\noindent \textbf{Gieseker-Maruyama moduli spaces.} Let $H$ be a polarization on a smooth projective surface $S$ satisfying $b_1(S) = 0$. Let $\rho \in \Z_{>0}$ and choose Chern classes $c_1 \in H^2(S,\Z)$ and $c_2 \in H^4(S,\Z) \cong \Z$. We denote by 
$$
M := M_S^H(\rho,c_1,c_2)
$$
the moduli space of rank $\rho$ Gieseker $H$-semistable torsion free sheaves $E$ on $S$ with $c_1(E) = c_1$ and $c_2(E) = c_2$. These moduli spaces were introduced by D.~Gieseker (surfaces) and M.~Maruyama (arbitrary dimension) \cite{Gie, Mar1, Mar2}, see also \cite{HL, Sim}. Gieseker-Maruyama moduli spaces generalize Hilbert schemes by the isomorphism
$$
S^{[n]} \cong M_S^H(\rho,0,n), \quad Z \mapsto [I_Z],
$$
where $I_Z \subset \O_S$ denotes the ideal sheaf of $Z \subset S$. The moduli space $M$ is a projective scheme and therefore provides an algebro-geometric compactification of the moduli space of rank $\rho$ Gieseker $H$-stable vector bundles on $S$ with Chern classes $c_1$ and $c_2$. 

In order to avoid complicated automorphism groups, we assume that all sheaves $E$ in $M$ are Gieseker $H$-stable. When (i) $K_S  H < 0$ and all sheaves in $M$ are $\mu$-stable or (ii) $K_S = 0$, the obstruction spaces $\Ext^2(E,E)_0$ vanish for all $[E] \in M$ and $M$ is smooth of expected dimension. In the smooth setting, the analog of \eqref{eulerrk1} has been studied in many cases (though mostly for ranks 2 and 3). A selection: \cite{Al1, Al2, Got1, Got2, Got4, GH, Kly, Koo, Man, Moz, Wei, Yos1, Yos2, Yos3}. For $S = \PP^2$, $\rho=2$, $c_1=H$, where $H$ is the class of a line, A.~Klyachko found the following formula using torus localization \cite{Kly}
\begin{equation} \label{Klyformula}
\sum_{c_2} e(M_{\PP^2}^H(2,H,c_2)) \, q^{c_2  - \frac{1}{2}} = 3 \eta(q)^{-6} \sum_{n=1}^{\infty} H(4n-1) q^{n - \frac{1}{4}},
\end{equation}
where $H(\Delta)$ is a Hurwitz class number. More precisely, $H(\Delta)$ denotes the number of (equivalence classes of) positive definite integral binary quadratic forms $AX^2 + BXY + CY^2$ with discriminant $-\Delta = B^2-4AC$ and weighted by the size of its automorphism group. By a result of D.~Zagier \cite{Zag}, this is a mock modular form of weight $-3/2$ in agreement with the S-duality predictions of C.~Vafa and E.~Witten \cite{VW}. In fact, when Vafa and Witten were writing their paper, \cite{Kly, Yos1} provided some of the few higher rank examples of such generating functions in the mathematics literature. 

In this survey, we are interested in smooth projective surfaces with holomorphic 2-form, i.e.~$p_g(S)>0$. Typically, these are surfaces of general type and their Gieseker-Maruyama moduli spaces are singular. (Although, for $c_2 \gg 0$, $M$ is irreducible and generically smooth of expected dimension, see \cite[Ch.~9]{HL} for references.) \\

\noindent \textbf{Virtual invariants.} The fundamental class of $M$ is in general out of reach. However, the moduli space $M$ carries a perfect obstruction theory in the sense of K.~Behrend and B.~Fantechi \cite{BF} or J.~Li and G.~Tian \cite{LT}. For Gieseker-Maruyama moduli spaces on surfaces, this was worked out by T.~Mochizuki \cite{Moc}. Then the virtual tangent bundle is given by
\begin{equation} \label{Tvir}
T_M^\vir = R\pi_* R\hom(\E,\E)_0[1],
\end{equation}
where $\E$ denotes the universal sheaf
on $S \times M$, $\pi : S \times M \rightarrow M$ is the  projection, and $(\cdot)_0$ denotes the trace-free part.\footnote{Although $\E$ may only exist \'etale locally, $R\pi_* R\hom(\E,\E)_0$ exists globally \cite[Thm.~2.2.4]{Cal}, see also \cite[Sect.~10.2]{HL}. Hence we do not need to assume $\E$ exists globally on $S \times M$.}  

This leads to a virtual class of degree equal to the expected dimension of $M$
\begin{equation} \label{vd}
[M]^{\vir} \in H_{2 \vd}(M), \quad \vd:=2rc_2-(r-1)c_1^2-(r^2-1)\chi(\O_S).
\end{equation}
One can now define the virtual Euler characteristic of $M$ by the virtual Poincar\'e-Hopf formula \cite{FG}
$$
e^{\vir}(M) := \int_{[M]^{\vir}} c_{\vd}(T_M^{\vir}).
$$
In Section \ref{sec:Segre:arbrk}, for any line bundle $L$ on $S$, we define the analog of $L^{[n]}$ for the Gieseker-Maruyama moduli space $M$, denoted by $L_M$, and we study the virtual Segre numbers
$$
\int_{[M]^{\vir}} s_{\vd}(L_M).
$$
We present a series of conjectures on virtual Euler characteristics and Segre numbers of Gieseker-Maruyama moduli spaces on arbitrary smooth projective surfaces $S$ satisfying $b_1(S) = 0$ and $p_g(S)>0$. More precisely, we will cover the following topics:
\begin{itemize}
\item \underline{Section \ref{sec:eulervir}}. Conjecture for virtual Euler characteristics of $M$ for rank $\rho=2$. Application to Vafa-Witten theory. Conjecture for virtual Euler characteristics of $M$ for $\rho=3$. Application to S-duality for $\rho=3$.
\item \underline{Section \ref{sec:refine}}. Conjecture for virtual $\chi_y$-genera of $M$ for $\rho=2$ and $3$. Application to $K$-theoretic S-duality conjecture. Conjecture for virtual elliptic genera of $M$ for $\rho=2$ (Dijkgraaf-Moore-Verlinde-Verlinde type formula). Conjecture for virtual cobordism classes of $M$ for $\rho=2$.
\item \underline{Section \ref{sec:verlinde}}. Conjectural Verlinde-type formula for $M$ for $\rho=2$. Application to a Verlinde-type formula for Higgs pairs on surfaces. Conjecture for virtual Verlinde numbers of $M$ in arbitrary rank. Conjecture motivated by virtual Serre duality.
\item \underline{Section \ref{sec:segre}}. Conjecture for virtual Segre numbers of $M$ in arbitrary rank. Application to a virtual Segre-Verlinde correspondence in arbitrary rank. 
\item \underline{Section \ref{sec:struc}}. Mochizuki's formula. Universal function. Verifications of conjectures.
\end{itemize}

\noindent \textbf{Quot schemes.} Instead of virtual invariants of Gieseker-Maruyama moduli spaces on surfaces, one can also consider virtual invariants of Quot schemes on surfaces. This has recently been explored in depth by Oprea-Pandharipande \cite{OP}, see also \cite{JOP,Lim}. This direction is currently attracting a lot of activity and leading to beautiful results. \\

\noindent \textbf{Acknowledgements.} Our work is influenced by many colleagues. We would like to thank A.~Gholampour, Y.~Jiang, T.~Laarakker, J.~Manschot, A.~Marian, H.~Nakajima, D.~Oprea, R.~Pandharipande, A.~Sheshmani, Y.~Tanaka, R.P.~Thomas, R.A.~Williams, S.-T.~Yau, and K.~Yoshioka. M.K~is supported by NWO grant VI.Vidi.192.012.

\section{Virtual Euler characteristics} \label{sec:eulervir}

\subsection{Rank 2} \label{sec:eulervir:rk2}

Let $(S,H)$ be a smooth polarized surface satisfying $b_1(S) = 0$ and $p_g(S)>0$. We denote by $\SW(a)$ the Seiberg-Witten invariant of $S$ in class $a \in H^2(S,\Z)$. Here we use Mochizuki's convention: $\SW(a) = \widetilde{\SW}(2a - K_S)$, where $\widetilde{\SW}(b)$ denotes the usual Seiberg-Witten invariant in class $b \in H^2(S,\Z)$ of differential geometry \cite[Sect.~6.3.2]{Moc}. We refer to $a$ as a Seiberg-Witten basic class when $\SW(a) \neq 0$. The Seiberg-Witten basic classes of $S$ are algebraic and satisfy $a(a-K_S) = 0$, i.e.~the virtual dimension of the linear system $|a|$ is zero. Moreover, Seiberg-Witten invariants satisfy the duality $\SW(a) = (-1)^{\chi(\O_S)} \SW(K_S-a)$. For $S$ minimal of general type, the Seiberg-Witten basic classes are $0$ and $K_S$, and $\SW(K_S) = (-1)^{\chi(\O_S)}$ \cite[Thm.~7.4.1]{Mor}.

We denote the normalized Dedekind eta function by $\overline{\eta}(x) = x^{-\frac{1}{24}} \eta(x)$. In order to formulate our first conjecture, we use the following theta functions
 \begin{align*}
 \theta_2(x) = \sum_{n \in \Z} x^{(n+\frac{1}{2})^2}, \quad \theta_3(x) = \sum_{n \in \Z} x^{n^2}.
\end{align*}

\begin{conjecture}\cite[Conj.~6.7]{GK1} \label{conj:evir:rk2}
Let $(S,H)$ be a smooth polarized surface satisfying $b_1(S) = 0$ and $p_g(S)>0$. Suppose $M:=M_S^H(2,c_1,c_2)$ contains no strictly Gieseker $H$-semistable sheaves. Then $e^{\vir}(M)$ equals the coefficient of $x^{\vd(M)}$ of 
\begin{align*}
4 \Bigg( \frac{1}{2 \overline{\eta}(x^2)^{12}} \Bigg)^{\chi(\O_S)}   \Bigg(\frac{2 \overline{\eta}(x^4)^2}{\theta_3(x)}\Bigg)^{K_{S}^2}  \sum_{a \in H^2(S,\Z)} \SW(a)(-1)^{a c_1} \Bigg(\frac{\theta_3(x)}{\theta_3(-x)}\Bigg)^{a K_S}.
\end{align*}
\end{conjecture}

For any smooth polarized surface $(S,H)$ satisfying $b_1(S)=0$, $\rho > 0$, and $c_1 \in H^2(S,\Z)$ such that $M_S^H(\rho,c_1,c_2)$ does not contain strictly Gieseker $H$-semistable sheaves for any $c_2$, we define the following generating function
\begin{equation*} \label{instantongenfun}
\sfZ_{S,H,\rho,c_1}^{\inst}(q) := \, q^{-\frac{1}{2\rho}\chi(\O_S) + \frac{\rho}{24} K_S^2} \sum_{c_2} e^{\vir}(M_S^H(\rho,c_1,c_2)) \, q^{\frac{\vd}{2\rho}},
\end{equation*}
where $\vd = \vd(M_S^H(\rho,c_1,c_2))$ is given by \eqref{vd}. With this normalization, it is not hard to show that Conjecture \ref{conj:evir:rk2} implies the following formula \cite[Eqn.~(29)]{GK1}
\begin{align*}
&\sfZ_{S,H,2,c_1}^{\inst}(q) =2\Bigg( \frac{1}{2 \Delta(q^{\frac{1}{2}})^{\frac{1}{2}}  } \Bigg)^{\chi(\O_S)} \Bigg( \frac{\theta_3(q)+\theta_2(q)}{2\eta(q)^2} \Bigg)^{-K_S^2} \sum_{a \in H^2(S,\Z)} \SW(a) \, (-1)^{ac_1} \, \Bigg( \frac{\theta_3(q)+\theta_2(q)}{\theta_3(q)-\theta_2(q)} \Bigg)^{aK_S} \\
&+2 i^{c_1^2} \Bigg(  \frac{1}{2\Delta(-q^{\frac{1}{2}})^{\frac{1}{2}}} \Bigg)^{\chi(\O_S)} \Bigg( \frac{\theta_3(q)+i\theta_2(q)}{2\eta(q)^2}\Bigg)^{-K_S^2} \sum_{a \in H^2(S,\Z)} \SW(a) \, (-1)^{ac_1} \, \Bigg( \frac{\theta_3(q)+i\theta_2(q)}{\theta_3(q)-i\theta_2(q)} \Bigg)^{aK_S},
\end{align*}
where $i = \sqrt{-1}$. In particular, the right-hand-side is independent of the polarization $H$ and only depends on $c_1$ modulo $2H^2(S,\Z)$.

\begin{remark} 
For surfaces with Seiberg-Witten basic classes $0$ and  $K_S \neq 0$, our conjecture for $\sfZ_{S,H,2,c_1}^{\inst}(q)$ coincides with line 2 of \cite[(5.38)]{VW}, i.e.~part of the contribution to the $\SU(2)$ Vafa-Witten partition function, which we discuss in Section \ref{sec:VW}. For arbitrary smooth polarized surfaces $(S,H)$ satisfying $b_1(S) = 0$ and $p_g(S) >0$, it coincides with terms two and three in \cite[Eqn.~(6.1)]{DPS} by R.~Dijkgraaf, J.-S.~Park, and B.J.~Schroers. As we will see in Section \ref{sec:VW}, our conjecture coincides with the instanton part of the $\SU(2)$ Vafa-Witten partition function.
\end{remark}

\begin{remark}
Conjecture \ref{conj:evir:rk2} implies a blow-up formula for virtual Euler characteristics \cite[Prop.~6.9]{GK1}. Surprisingly, it is identical to the blow-up formula for topological Euler characteristics derived in \cite[Prop.~3.1]{Got4}, \cite{Yos1}. Our conjectures for rank 3 virtual Euler characteristics and rank 2 and 3 virtual $\chi_{y}$-genera, discussed later in this survey, yield blow-up formulae which are also identical to those of their ``motivic counterparts''. More precisely, for the case of virtual $\chi_y$-genera (virtual in the sense of virtual classes), we get the same blow-up formula as the one for ``motivic'' $\chi_y$-genera which follows from the work of W.-P.~Li and Z.~Qin \cite{LQ1, LQ2}. Proving the blow-up formula for virtual Euler characteristics and virtual $\chi_y$-genera is an interesting open problem.
\end{remark}

When $S$ is a K3 surface and assuming ``stable equals semistable'', $M_S^H(\rho, c_1, c_2)$ is smooth of expected dimension $\vd$ and deformation equivalent to $S^{[\vd/2]}$ \cite{OG, Huy, Yos4}. Therefore Conjecture \ref{conj:evir:rk2} holds for K3 surfaces by \eqref{eulerrk1}. In addition, consider the following list of surfaces: \\

\noindent elliptic surfaces of type\footnote{An elliptic surface of type $E(n)$ is an elliptic surface $S \rightarrow \PP^1$ with section, $12n$ rational 1-nodal fibres, and no further singular fibres.} $E(n)$ with $n \in \{3,4,5,6,7,8\}$, blow-ups of a K3 surface in one or two points, double covers of $\PP^2$ branched along a smooth curve of degree $2d$ with $d \in\{4,5,6,7\}$, certain double covers of $\PP^1 \times \PP^1$ and the Hirzebruch surfaces $\FF_1, \FF_2, \FF_3$ \cite[Sect.~7.4]{GK1}, smooth quintics and sextics in $\PP^3$, smooth surfaces of bidegree $(4,3)$, $(5,3)$, $(6,3)$, $(4,4)$, $(5,4)$, $(4,5)$ in $\PP^2 \times \PP^1$, smooth surfaces of tridegree $(3,3,3)$, $(3,3,4)$, $(3,3,5)$, $(3,4,4)$ in $\PP^1 \times \PP^1 \times \PP^1$, smooth complete intersections of hypersurfaces of degrees 2 and 4, or 2 and 5, or 3 and 3, or 3 and 4 in $\PP^4$, smooth complete intersections of hypersurfaces of degrees 2 and 2 and 3 in $\PP^5$. \\

In each of these cases, and for certain values of $c_1$, we verified Conjecture \ref{conj:evir:rk2} for Gieseker-Maruyama moduli spaces up to high virtual dimension. The precise upper bound on virtual dimension, up to which we verified Conjecture \ref{conj:evir:rk2}, depends on the case and usually lies between $25$ and $70$. 

These verifications rely on a certain universal function (Theorem \ref{thm:univ}), which we derived from Mochizuki's formula as described in detail in Section \ref{sec:struc}.

\subsection{Application: Vafa-Witten invariants} \label{sec:VW}

In 1994, Vafa-Witten proposed new tests for S-duality of $N=4$ supersymmetric Yang-Mills theory on a real 4-manifold $M$ \cite{VW}. This theory involves a gauge group, denoted by $G$, and coupling constants $\theta$ and $g$ grouped into a complex parameter
$$
\tau := \frac{\theta}{2 \pi} + \frac{4 \pi i}{g^2}. 
$$
Suppose $M$ underlies a complex smooth projective surface $S$ and $G$ equals $\SU(\rho)$ or its Langlands dual $\SU(\rho) / \Z_\rho$. After topological twisting, Vafa-Witten argued that S-duality implies that the partition functions satisfy
\begin{equation} \label{Sdualityweak}
\sfZ_{\SU(\rho)}(-1/\tau) =  (-1)^{(\rho-1)\chi(\O_S)} \Big( \frac{\rho \tau}{i} \Big)^{-\frac{e(S)}{2}} \sfZ_{\SU(\rho) / \Z_\rho}(\tau).
\end{equation}
Roughly speaking: the theory for gauge group $\SU(\rho)$ and ``strong coupling $-1/\tau$'' is equivalent to the theory for Langlands dual gauge group $\SU(\rho) / \Z_\rho$ at ``weak coupling $\tau$''.
 Referring in parts to the mathematics literature \cite{Kly, Yos1, Yos2, Nak2, Nak3}, Vafa-Witten performed some non-trivial modularity checks for $S = \PP^2$ (using \eqref{Klyformula}), K3, blow-ups, and ALE spaces (mostly for rank $\rho=2$). 

In \cite[Sect.~5]{VW}, using superconducting cosmic strings, Vafa-Witten predicted a formula for the partition function, when $S$ is a smooth projective surface having a connected smooth curve in $|K_S|$. Their formula was generalized to arbitrary smooth projective surfaces $S$ satisfying $p_g(S)>0$ in \cite{DPS}. At the time, there existed no mathematical verifications, or even a definition, of the Vafa-Witten partition function for this setting. For $S$ any smooth projective surface satisfying $H_1(S,\Z) = 0$ and $p_g(S)>0$, and arbitrary $c_1$, the formula predicted by physics is as follows \cite[Eqn.~(6.1)]{DPS}:
\begin{align}
\begin{split} \label{(5.38)}
&\sfZ_{S,H,2,c_1}(q) = \Bigg( \frac{1}{2\Delta(q^2)^{\frac{1}{2}}} \Bigg)^{\chi(\O_S)} \Bigg( \frac{\theta_3(q)}{\eta(q)^2} \Bigg)^{-K_S^2} (-1)^{\chi(\O_S)} \sum_{a \in H^2(S,\Z)} \SW(a) \, \delta_{a, c_1} \,  \Bigg( \frac{\theta_3(q)}{\theta_2(q)} \Bigg)^{a K_S} \\
&+2\Bigg( \frac{1}{2 \Delta(q^{\frac{1}{2}})^{\frac{1}{2}}  } \Bigg)^{\chi(\O_S)} \Bigg( \frac{\theta_3(q)+\theta_2(q)}{2\eta(q)^2} \Bigg)^{-K_S^2} \sum_{a \in H^2(S,\Z)} \SW(a) \, (-1)^{ac_1} \, \Bigg( \frac{\theta_3(q)+\theta_2(q)}{\theta_3(q)-\theta_2(q)} \Bigg)^{aK_S} \\
&+2 i^{c_1^2} \Bigg(  \frac{1}{2\Delta(-q^{\frac{1}{2}})^{\frac{1}{2}}} \Bigg)^{\chi(\O_S)} \Bigg( \frac{\theta_3(q)+i\theta_2(q)}{2\eta(q)^2}\Bigg)^{-K_S^2} \sum_{a \in H^2(S,\Z)} \SW(a) \, (-1)^{ac_1} \, \Bigg( \frac{\theta_3(q)+i\theta_2(q)}{\theta_3(q)-i\theta_2(q)} \Bigg)^{aK_S},
\end{split}
\end{align}
where $q = e^{2\pi i \tau}$, $\tau \in \mathfrak{H}$ (the upper half plane), and for any $a,b \in H^2(S,\Z)$
\begin{equation*} \label{deltadef}
\delta_{a,b} := \left\{\begin{array}{cc} 1 & \qquad \textrm{if \ } a-b \in 2 H^2(S,\Z)  \\ 0 & \qquad \textrm{otherwise}.  \end{array}\right.
\end{equation*}
Terms two and three of \eqref{(5.38)} coincide with our 
conjecture for virtual Euler characteristics $\sfZ_{S,H,2,c_1}^{\inst}(q)$ of the previous section. In Section \ref{sec:rk3Sduality}, we discuss in which sense \eqref{(5.38)} satisfies the S-duality transformation \eqref{Sdualityweak}. In this section, we focus on equation \eqref{(5.38)} itself.

Around the time we were working on \cite{GK1}, Y.~Tanaka and R.P.~Thomas \cite{TT1} discovered the mathematical definition of $\SU(\rho)$ Vafa-Witten invariants using a symmetric perfect obstruction theory on the moduli space of Higgs pairs on $S$. Also around that time, A.~Gholampour, A.~Sheshmani and S.-T.~Yau \cite{GSY2} were studying certain reduced Donaldson-Thomas invariants of the non-compact Calabi-Yau threefold $X = \mathrm{Tot}(K_S)$, which (up to an equivariant parameter) are equal to Tanaka-Thomas's invariants. We briefly describe both works. 

Let $(S,H)$ be a smooth polarized surface and $L \in \Pic(S)$. Tanaka-Thomas \cite{TT1} consider the moduli space of isomorphism classes of $H$-semistable Higgs pairs
$$
N^{\perp} := N_S^H(\rho,L,c_2) = \big\{ [(E,\phi)] \, : \, \det E \cong L, \ \tr \phi = 0, \ c_2(E) =c_2 \big\},
$$
where $E$ is a rank $\rho$ torsion free sheaf, $\phi : E \rightarrow E \otimes K_S$ is a trace-free morphism, and the pair $(E,\phi)$ satisfies a (Gieseker) semistability condition with respect to $H$. Assuming ``stable equals semistable'', Tanaka-Thomas show that $N^{\perp}$ admits a symmetric perfect obstruction theory (symmetric in the sense of Behrend \cite{Beh}). The $\C^*$-scaling action on the Higgs field lifts to $N^\perp$. Although $N^\perp$ is not proper, its fixed locus $(N^\perp)^{\C^*}$ is proper and Tanaka-Thomas define $\SU(\rho)$ Vafa-Witten invariants by 
\begin{equation} \label{defloc}
\VW_S^H(\rho,L,c_2) := \int_{[N_S^H(\rho,L,c_2)^{\C^*}]^{\vir}} \frac{1}{e(N^\vir)} \in \Q,
\end{equation}
which is the virtual localization formula of T.~Graber and Pandharipande \cite{GP}. In particular, $e(N^{\vir})$ denotes the equivariant Euler class of the virtual normal bundle to $(N^\perp)^{\C^*}$. There are two types of components of $(N^\perp)^{\C^*}$. Higgs pairs with $\phi = 0$ form a component isomorphic to the Gieseker-Maruyama moduli space $M:=M_S^H(\rho,L,c_2)$ (the instanton branch). Tanaka-Thomas show that the contribution of $M$ to \eqref{defloc} is
$$
(-1)^{\vd(M)} e^\vir(M) \in \Z.
$$ 
We refer to the other components of $(N^\perp)^{\C^*}$ as the monopole branch. When (i) $K_S  H < 0$ and all sheaves in $M$ are $\mu$-stable or (ii) $K_S = 0$, 
there are no contributions from the monopole branch to \eqref{defloc}, $M$ is smooth of expected dimension, and $e^{\vir}(M) = e(M)$ \cite[Prop.~7.4]{TT1}. For surfaces containing a connected smooth canonical curve, Tanaka-Thomas calculated the contribution of the monopole branch for $\rho=2$ and $c_2 \leq 3$, and obtained a match with the first term of \eqref{(5.38)}. Together with Conjecture \ref{conj:evir:rk2}, this provides compelling evidence that Tanaka-Thomas found the right mathematical definition of the $\SU(\rho)$ Vafa-Witten generating function, i.e.
\begin{equation} \label{mathdefVW}
\sfZ_{S,H,\rho,c_1}(q) := q^{-\frac{1}{2\rho}\chi(\O_S) + \frac{\rho}{24} K_S^2} \sum_{c_2} (-1)^{\vd} \, \VW_S^H(\rho,c_1,c_2) \, q^{\frac{\vd}{2\rho}},
\end{equation}
where $\vd := 2\rho c_2-(\rho-1)c_1^2-(\rho^2-1)\chi(\O_S)$.

\begin{remark}
Initially, Tanaka-Thomas proposed two candidate definitions for $\SU(\rho)$ Vafa-Witten invariants \cite{TT1}. Their second definition is by integrating Behrend's constructible function over $N^\perp$. Since $N^\perp$ is non-proper, this definition is in general not equal to the above definition using virtual classes (and in fact produces the ``wrong'' numbers from the point of view of physics). As recounted in the introduction of \cite{TT1}, Conjecture \ref{conj:evir:rk2} played a crucial role in the realization that definition \eqref{defloc} is the correct one.
\end{remark}

The components of the Higgs moduli space $(N^\perp)^{\C^*}$ can be indexed by the ranks of the eigensheaves 
$$
E = \bigoplus_i E_i \otimes \mathfrak{t}^{-i}
$$
of $[(E,\phi)] \in (N^\perp)^{\C^*}$, where $\mathfrak{t}$ denotes a degree one character of $\C^*$. The following theorem of T.~Laarakker \cite{Laa1} deals with the components indexed by eigenrank $(1, \ldots, 1)$.
\begin{theorem}[Laarakker] \label{thm:Laa}
Fix $\rho>1$. Then there exist $A,C_{ij} \in \Q(\!(q^{\frac{1}{2\rho}})\!)$, for all $1 \leq i \leq j \leq \rho-1$, and $B \in q^{\frac{\rho}{24}} \, \Q(\!(q^{\frac{1}{2\rho}})\!)$ with the following property.\footnote{We suppress the dependence of these universal functions on $\rho$.} Let $(S,H)$ be a smooth polarized surface satisfying $H_1(S,\Z) = 0$ and $p_g(S)>0$.\footnote{After normalizing by the order of the $\rho$-torsion subgroup of $H^2(S,\Z)$, Laarakker's result holds without the condition $H_1(S,\Z)=0$ \cite{Laa1}.} Suppose $H, \rho, c_1$ are chosen such that $N_S^H(\rho,c_1,c_2)$ does not contain strictly Gieseker $H$-semistable Higgs pairs for any $c_2$. Then the contribution of Higgs pairs with eigenrank $(1, \ldots, 1)$ to $\sfZ_{S,H,\rho,c_1}(q)$ is given by
$$
A^{\chi(\O_S)} B^{K_S^2} \sum_{(a_1, \ldots, a_{\rho-1})} \prod_{i=1}^{\rho-1} \SW(a_i) \prod_{1 \leq i \leq j \leq \rho-1} C_{ij}^{a_i a_j},
$$
where the sum is over all $(a_1, \ldots, a_{\rho-1}) \in H^2(S,\Z)^{\rho-1}$ satisfying
$$
c_1 - \sum_{i=1}^{\rho-1} i a_i \in \rho H^2(S,\Z).
$$
\end{theorem}
The proof of Laarakker's theorem relies on a beautiful description, by Gholampour-Thomas \cite{GT1, GT2}, of the components of $(N^\perp)^{\C^*}$ indexed by $(1, \ldots, 1)$ in terms of nested Hilbert schemes. Consequently, the universal functions $A(q),B(q),C_{ij}(q)$ can be expressed in terms of intersection numbers on products of Hilbert schemes of points on $S$. These can be determined (up to some order in $q$) by toric calculations similar to ours discussed in Section \ref{sec:toric}. As an application, Laarakker calculated the first 15 non-zero terms of the monopole contribution to $\sfZ_{S,H,2,c_1}(q)$ and found agreement with \eqref{(5.38)}.
In Section \ref{sec:rk3Sduality}, we discuss an application of Theorem \ref{thm:Laa} to $\mathrm{SU}(3)$ Vafa-Witten invariants.

\begin{remark}
In \cite{TT2}, Tanaka-Thomas removed the ``stable equals semistable'' assumption by using Joyce-Song pairs. Using their definition of generalized Vafa-Witten invariants, Laarakker \cite{Laa2} showed that the ``stable equals semistable'' condition can be dropped from Theorem \ref{thm:Laa}, as expected from physics predictions.
\end{remark}

As mentioned in the beginning of this section, Gholampour-Sheshmani-Yau \cite{GSY1, GSY2} provided an interpretation of Vafa-Witten invariants in terms of reduced Donaldson-Thomas invariants of the non-proper Calabi-Yau threefold $X = \mathrm{Tot}(K_S)$ when $b_1(S)=0$. They consider the moduli space
$$
M_X:=M_X^H(\mathrm{ch})
$$
of pure dimension 2 Gieseker $H$-stable sheaves on $X$ with proper support and Chern character
$$
\mathrm{ch} = (0,\rho[S], \mathrm{ch}_2, \mathrm{ch}_3).
$$
The moduli space $M_X$ admits a symmetric perfect obstruction theory by \cite{Tho} and Gholampour-Sheshmani-Yau reduce this perfect obstruction theory by taking out a trivial piece of rank $p_g(S)$ from the obstruction bundle (similar to \cite{BL, KT1, KT2, Lee, Li} for Gromov-Witten and Pandharipande-Thomas invariants in various settings). The moduli space $M_X$ has a $\C^*$-action induced by the natural $\C^*$-action on the fibres of $X$. Furthermore, $M_X^{\C^*} \cong (N^\perp)^{\C^*}$ and, after restriction to the fixed locus, the $\C^*$-fixed parts of $T_{M_X}^{\vir}$ and $T_{N^\perp}^{\vir}$ are equal (in $K$-theory). Since their $\C^*$-moving parts  
only differ by a trivial piece, Gholampour-Sheshmani-Yau's invariants are equal to $\VW_S^H(\rho,L,c_2)$ (up to an equivariant parameter).

\subsection{Rank 3} \label{sec:eulervir:rk3}

In this section, we present a conjecture for the virtual Euler characteristics of rank 3 Gieseker-Maruyama moduli spaces. Consider the $A_2$-lattice consisting of $\Z^2$ together with bilinear form $\langle v , w \rangle := v^t A w$ given by
\begin{equation*} \label{matrixA}
A= \left(\begin{array}{cc} 2 & -1 \\ -1 & 2 \end{array}\right).
\end{equation*}
The dual lattice $A_2^\vee$ is given by $\Z^2$ and $\langle v , w \rangle^\vee := v^t A^\vee w$ where
\begin{equation*} \label{matrixAinv}
A^\vee = A^{-1} =  \frac{1}{3} \left(\begin{array}{cc} 2 & 1 \\ 1 & 2 \end{array}\right).
\end{equation*}
Let $\epsilon := e^{\frac{2 \pi i}{3}}$. We will use the following theta functions
\begin{align*}
\Theta_{A_2,(0,0)}(x) :=& \, \sum_{v \in \Z^2} (x^2)^{\frac{1}{2} \langle v,v\rangle} =\, \sum_{(m,n) \in \Z^2} x^{2(m^2 -mn +n^2)}, \\
\Theta_{A_2,(1,0)}(x) :=& \, \sum_{v \in \Z^2} (x^2)^{\frac{1}{2}\langle v+(\frac{1}{3},-\frac{1}{3}),v +(\frac{1}{3},-\frac{1}{3}) \rangle} =\, \sum_{(m,n) \in \Z^2} x^{2(m^2 -mn +n^2+m - n +\frac{1}{3})}, \\
\Theta_{A_2^\vee,(0,0)}(x) :=& \, \sum_{v \in \Z^2} (x^6)^{\frac{1}{2} \langle v,v\rangle^\vee} = \, \sum_{(m,n) \in \Z^2} x^{2(m^2 +mn +n^2)},  \\
\Theta_{A_2^\vee,(0,1)}(x) :=& \, \sum_{v \in \Z^2} (x^6)^{\frac{1}{2} \langle v,v\rangle^\vee} e^{2 \pi i \langle v,(1,-1) \rangle^\vee} =\,\sum_{(m,n) \in \Z^2} \epsilon^{m-n} x^{2(m^2 +mn +n^2)}.
\end{align*}
\begin{conjecture} \cite[Conj.~1.1]{GK3} \label{conj:evir:rk3}
Let $(S,H)$ be a smooth polarized surface satisfying $b_1(S) = 0$ and $p_g(S)>0$. Suppose $M:=M_S^H(3,c_1,c_2)$ contains no strictly Gieseker $H$-semistable sheaves. Then $e^{\vir}(M)$ equals the coefficient of $x^{\vd(M)}$ of \begin{align*}
&9 \Bigg( \frac{1}{3 \overline{\eta}(x^2)^{12}} \Bigg)^{\chi(\O_S)} \Bigg( \frac{\Theta_{A^\vee_2, (0,1)}(x)}{3 \overline{\eta}(x^6)^3} \Bigg)^{-K_S^2} \sum_{(a,b)} \SW(a) \, \SW(b) \, \epsilon^{(a-b)c_1} \, Z_{+}(x)^{ab} \, Z_{-}(x)^{(K_S-a)(K_S-b)}, 
\end{align*}
where the sum is over all $(a,b) \in H^2(S,\Z) \times H^2(S,\Z)$ and $Z_{\pm}(x)$ are the solutions to the following quadratic equation in $\zeta$
\begin{align*}
\zeta^2 - 4Z(x)^2 \, \zeta + 4Z(x) = 0,
\end{align*}
where $Z(x) := \frac{\Theta_{A^\vee_2, (0,0)}(x)}{\Theta_{A^\vee_2, (0,1)}(x)}$.
\end{conjecture}

As in the rank 2 case, Conjecture \ref{conj:evir:rk3} holds for K3 surfaces by deformation invariance and \eqref{eulerrk1}. Moreover, we consider the following list of surfaces: \\

\noindent elliptic surfaces of type $E(3), E(4), E(5)$, blow-ups of an elliptic surface of type $E(3)$ in one point, blow-ups of a K3 surface in one or two points, double covers of $\PP^2$ branched along a smooth octic, blow-ups of the previous double covers in one point, double covers of $\PP^1 \times \PP^1$ branched along a smooth curve of bidegree $(6,6)$, blow-ups of the previous double covers in one point, smooth quintics in $\PP^3$, blow-ups of a smooth quintic in $\PP^3$ in one point, certain surfaces with small values of $p_g(S)$ and $K_S^2$ constructed by Kanev, Catanese-Debarre, and Persson \cite[Sect.~2.4]{GK3}. \\

In each case, we verified Conjecture \ref{conj:evir:rk3} for Gieseker-Maruyama moduli spaces of certain virtual dimensions, considerably lower than in the rank 2 case, and for several choices of $c_1$.  
The precise list of verifications can be found in \cite[Sect.~2.4]{GK3}. As in the rank 2 case, our method for these verifications is discussed in Section \ref{sec:struc}.

\subsection{Application: S-duality in rank 3} \label{sec:rk3Sduality}

Formula \eqref{(5.38)} for $\SU(2)$ Vafa-Witten invariants was already known to physicists in 1994 \cite{VW}. One may wonder whether the recent mathematical developments in Vafa-Witten theory led to the discovery of any new formulae. Two new directions are:
\begin{itemize}
\item A new conjectural formula for the $\SU(3)$ Vafa-Witten invariants, which we describe in this section.
\item Refinements of Vafa-Witten invariants, which are discussed in Section \ref{sec:refine}.
\end{itemize}

\begin{remark}
In the physics literature, there exists a formula for the $\SU(\rho)$ Vafa-Witten invariants for any prime rank $\rho$ and any smooth projective surface $S$ satisfying $H_1(S,\Z) = 0$ and containing a smooth connected canonical curve, i.e.~\cite[Eqn.~(5.13)]{LL}. This formula appears incorrect. Take $S$ an elliptic surface of type $E(3)$, $\rho=3$, $c_1 = B$, where $B$ is the class of a section, $c_2=3$, and a suitable polarization $H$. Then a result of T.~Bridgeland \cite{Bri} implies that $M:=M_S^H(3,B,3)$ is smooth of expected dimension and consists of a single reduced point. Hence $e^{\vir}(M) = e(M) = 1$, which does not match the instanton part of \cite[Eqn.~(5.13)]{LL}. 
\end{remark}

Let $S$ be a smooth projective surface satisfying $H_1(S,\Z) = 0$ and $p_g(S)>0$, and consider the $\SU(\rho)$ Vafa-Witten partition function $\sfZ_{S,H,\rho,c_1}(q)$ defined in \eqref{mathdefVW}. 
Let $\rho=1$ or $\rho$ prime. Vafa-Witten predicted that $\sfZ_{S,H,\rho,c_1}(q)$ only depends on $[c_1] \in H^2(S,\Z) / \rho H^2(S,\Z)$ and is the Fourier expansion of a meromorphic function $\sfZ_{S,H,\rho,c_1}(\tau)$ on $\mathfrak{H}$ satisfying
 \cite[(5.39)]{VW}, \cite[(5.22)]{LL}
\begin{align} \label{Sdualitystrong}
\begin{split}
\sfZ_{S,H,\rho,c_1}(\tau+1) &= (-1)^{\rho \chi(\O_S)} \, e^{\frac{\pi i \rho}{12} K_S^2} \, e^{-\frac{\pi i (\rho-1)}{\rho} c_1^2} \, \sfZ_{S,H,\rho,c_1}(\tau), \\
\sfZ_{S,H,\rho,c_1}(-1/\tau) &= (-1)^{(\rho-1)\chi(\O_S)} \, \rho^{1-\frac{e(S)}{2}} \, \Big( \frac{\tau}{i} \Big)^{-\frac{e(S)}{2}} \, \sum_{[a]} e^{\frac{2 \pi i}{\rho} a c_1} \sfZ_{S,H,\rho,a}(\tau),
\end{split}
\end{align}
where the sum is over all $[a] \in H^2(S,\Z) / \rho H^2(S,\Z)$. 
S-duality transformation \eqref{Sdualitystrong} implies \eqref{Sdualityweak} as follows. Define
\begin{align*}
\sfZ_{\SU(\rho)} := \rho^{-1} \sfZ_{S,H,\rho,0}, \quad \sfZ_{\SU(\rho) / \Z_\rho} :=\sum_{[a]} \sfZ_{S,H,\rho,a},
\end{align*}
where the sum is over all $[a] \in H^2(S,\Z) / \rho H^2(S,\Z)$ and $\sfZ_{S,H,\rho,a}$ was defined in \eqref{mathdefVW}. Taking $c_1=0$, \eqref{Sdualitystrong} implies \eqref{Sdualityweak}. 

\begin{remark} \label{non-alg}
There is an important subtlety in the previous discussion. By definition \eqref{mathdefVW}, the generating function $\sfZ_{S,H,\rho,c_1}(q)$ is obviously zero for non-algebraic classes $c_1 \in H^2(S,\Z)$. For the above discussion to make sense, we need the following more precise formulation: conjecturally there exists a series $\widetilde{\sfZ}_{S,H,\rho,c_1}(q)$ defined for any $S,H,\rho$ as above and any possibly non-algebraic $c_1 \in H^2(S,\Z)$ such that:
\begin{itemize}
\item $\widetilde{\sfZ}_{S,H,\rho,c_1}(q)$ only depends on $[c_1] \in H^2(S,\Z) / \rho H^2(S,\Z)$,
\item $\widetilde{\sfZ}_{S,H,\rho,c_1}(q) = \sfZ_{S,H,\rho,c_1}(q)$ for algebraic classes $c_1 \in H^2(S,\Z)$,
\item $\widetilde{\sfZ}_{S,H,\rho,c_1}(q)$ is the Fourier expansion of a meromorphic function $\widetilde{\sfZ}_{S,H,\rho,c_1}(\tau)$ on $\mathfrak{H}$ satisfying \eqref{Sdualitystrong}.
\end{itemize}
Clearly it is desirable to have a geometric definition of $\sfZ_{S,H,\rho,c_1}(q)$ for non-algebraic classes $c_1 \in H^2(S,\Z)$. Y.~Jiang \cite{Jia} recently introduced important ideas for such a definition by considering the Vafa-Witten theory of $\mu_\rho$-gerbes. In \cite{JK}, the second-named author and Jiang give a mathematical definition of the $\SU(\rho) / \Z_{\rho}$ Vafa-Witten partition function, using K.~Yoshioka's moduli spaces of twisted sheaves \cite{Yos5}, and prove the S-duality conjecture for K3 surfaces and arbitrary prime rank $\rho$.
\end{remark} 

The instanton contribution $\sfZ^{\inst}_{S,H,3,c_1}(q)$ to $\sfZ_{S,H,3,c_1}(q)$ is predicted by Conjecture \ref{conj:evir:rk3}. Combined with the physicists' S-duality prediction \eqref{Sdualitystrong}, we conjectured \cite[Conj.~1.5]{GK3} the following formula for the monopole contribution $\sfZ^{\mono}_{S,H,3,c_1}(q) := \sfZ_{S,H,3,c_1}(q) - \sfZ^{\inst}_{S,H,3,c_1}(q)$ 
\begin{align} \label{conj:mono:rk3}
\Bigg( \frac{1}{3\Delta(q^3)^{\frac{1}{2}}} \Bigg)^{\chi(\O_S)} \Bigg(\frac{\Theta_{A_2,(1,0)}(q^{\frac{1}{2}})}{\eta(q)^3}  \Bigg)^{-K_S^2} \sum_{(a,b)} \SW(a) \, \SW(b) \,\delta_{c_1+a,b} \, W_+(q^{\frac{1}{2}})^{ab} \, W_-(q^{\frac{1}{2}})^{(K_S-a)(K_S-b)}, 
\end{align}
where the sum is over all $(a,b) \in H^2(S,\Z) \times H^2(S,\Z)$. Moreover, $W_{\pm}(x)$ are the solutions of the following quadratic equations in $\omega$
\begin{align*}
\omega^2 - 4W(x)^2 \, \omega + 4W(x) = 0,
\end{align*}
where $W(x) := \frac{\Theta_{A_2, (0,0)}(x)}{\Theta_{A_2, (1,0)}(x)}$. Note that the instanton contribution (Conjecture \ref{conj:evir:rk3}) involves the theta function of the lattice $A_2^\vee$, whereas the monopole contribution \eqref{conj:mono:rk3} involves the theta function of the lattice $A_2$. Similarly, one can write the instanton and monopole part of \eqref{(5.38)} in terms of the theta function of the $A_1^\vee$-lattice and $A_1$-lattice respectively. 

We now discuss some remarkable verifications of Conjecture \eqref{conj:mono:rk3}. Recall that the components of $(N^\perp)^{\C^*}$ can be indexed by eigenrank (Section \ref{sec:VW}). Using cosection localization \cite{KL1, KL2}, Thomas \cite[Thm.~5.23]{Tho} proved the following powerful theorem.
\begin{theorem}[Thomas] \label{thm:Thomas}
Let $S$ be a smooth projective surface satisfying $p_g(S)>0$ and let $\rho$ be prime. Suppose $N_S^H(\rho,L,c_2)$ does not contain strictly Gieseker $H$-semistable Higgs pairs for any $c_2$. Then only Higgs pairs with eigenranks $(\rho)$ and $(1, \ldots, 1)$ contribute to $\sfZ_{S,H,\rho,c_1}(q)$.
\end{theorem}
The component of $N_S^H(\rho,L,c_2)^{\C^*}$ corresponding Higgs pairs with eigenrank $(\rho)$ is precisely the Gieseker-Maruyama moduli space $M_S^H(\rho,L,c_2)$ (in this case, the Higgs field $\phi=0$). By Thomas's theorem, all Higgs pairs contributing to $\sfZ^{\mono}_{S,H,3,c_1}(q)$ have eigenrank $(1,1,1)$. Using Theorems \ref{thm:Laa} and \ref{thm:Thomas}, Laarakker proved that the first 11 non-zero coefficients of $\sfZ^{\mono}_{S,H,3,c_1}(q)$ are indeed as predicted by Conjecture \eqref{conj:mono:rk3}. It is worth noting that, for prime rank, calculations on the monopole branch are easier than their analogs on the instanton branch (essentially because Theorem \ref{thm:Laa} does not involve taking residues as opposed to our universality results such as Theorem \ref{thm:univ} described in Section \ref{sec:struc}). 

In \cite{GK3}, we proved the following result.
\begin{theorem} \cite[Prop.~4.10]{GK3}
The conjectural formula for $\sfZ_{S,H,3,c_1}(q)$, determined by Conjectures \ref{conj:evir:rk3} and \eqref{conj:mono:rk3}, satisfies the S-duality transformation \eqref{Sdualitystrong}.
\end{theorem}

The proof combines properties of quite diverse mathematical objects: Seiberg-Witten invariants, the lattice $(H^2(S,\Z), \cup)$, Gauss sums and Dedekind sums, and lattice theta functions.

\begin{remark}
Roughly speaking, the S-duality transformation \eqref{Sdualitystrong} swaps the contributions of the monopole and instanton branch. We do not know what this duality corresponds to geometrically. It is highly remarkable that, for prime rank, our ``non-abelian'' calculations on the instanton branch appear to contain the same information as Laarakker's ``abelian'' calculations on the monopole branch.
\end{remark}

\section{Refinements} \label{sec:refine}

The method we used for our verifications of Conjectures \ref{conj:evir:rk2} and \ref{conj:evir:rk3} holds quite generally; not just for virtual Euler characteristics (Section \ref{sec:struc}). This allowed us to find refinements to virtual $\chi_y$-genus, elliptic genus, and cobordism class.

\subsection{Virtual $\chi_y$-genera}

The normalised virtual $\chi_y$-genus of a proper $\C$-scheme $Z$, with perfect obstruction theory with virtual tangent bundle $T_{Z}^{\vir}$, is defined by \cite{FG}
\begin{align*}
\overline{\chi}_{-y}^{\vir}(Z) &:= y^{-\frac{\vd(Z)}{2}}\chi_{-y}^{\vir}(Z), \\
\chi_{y}^{\vir}(Z) &:=  \sum_{p \geq 0} y^p \, \chi(Z,\Lambda^p \Omega_{Z}^{\vir} \otimes \O_Z^{\vir}) \in \Z[y],
\end{align*}
where $\vd(Z) := \rk T_Z^{\vir}$, $\O_Z^{\vir}$ denotes the virtual structure sheaf of $Z$ and $\Omega_{Z}^{\vir} := (T_{Z}^{\vir})^{\vee}$. The normalized virtual $\chi_y$-genus is a symmetric Laurent polynomial in $y^{\frac{1}{2}}$ by \cite[Cor.~4.9]{FG}. Moreover, $e^\vir(Z) =\overline \chi_{-1}^{\vir}(Z)$.

In order to formulate the analogs of Conjectures \ref{conj:evir:rk2} and \ref{conj:evir:rk3} for virtual $\chi_y$-genera, we require the following refinements of
the lattice theta functions of the previous section
 \begin{align*}
\theta_2(x,y) = \sum_{n \in \Z} x^{(n+\frac{1}{2})^2} y^{n+\frac{1}{2}}, \quad \theta_3(x,y) = \sum_{n \in \Z} x^{n^2} y^{n}
\end{align*}
and
\begin{align*}
\Theta_{A_2,(0,0)}(x,y) &:= \, \sum_{(m,n) \in \Z^2} x^{2(m^2 -mn +n^2)} y^{m+n}, \quad \Theta_{A_2,(1,0)}(x,y) := \, \sum_{(m,n) \in \Z^2} x^{2(m^2 -mn +n^2+m - n +\frac{1}{3})} y^{m+n}, \\
\Theta_{A_2^\vee,(0,0)}(x,y) &:= \, \sum_{(m,n) \in \Z^2} x^{2(m^2 +mn +n^2)} y^{m+n},  \quad \Theta_{A_2^\vee,(0,1)}(x,y) := \,\sum_{(m,n) \in \Z^2} \epsilon^{m-n} x^{2(m^2 +mn +n^2)} y^{m+n},
\end{align*}
where $\epsilon = e^{\frac{2 \pi i}{3}}$. The analog of Conjecture \ref{conj:evir:rk2} for virtual $\chi_y$-genus is straight-forward.
\begin{conjecture} \cite[Conj.~6.7]{GK1} \label{conj:chiyvir:rk2}
The statement of Conjecture \ref{conj:evir:rk2} holds with the following replacements: $e^{\vir}$ replaced by $\overline{\chi}_{-y}^{\vir}$, $\overline{\eta}(x^2)^{12}$ replaced by 
$$
\prod_{n=1}^{\infty} (1-x^{2n}y) (1-x^{2n}y^{-1}) (1-x^{2n})^{10},
$$
and $\theta_3(x)$ replaced by $\theta_3(x,y^{\frac{1}{2}})$. Note that $\overline{\eta}(x^4)^2$ does not get replaced.
\end{conjecture}

The analog of Conjecture \ref{conj:evir:rk3} for virtual $\chi_y$-genus involves a surprising refinement of the quadratic equation.
\begin{conjecture}  \cite[Conj.~1.1]{GK3}  \label{conj:chiyvir:rk3}
The statement of Conjecture \ref{conj:evir:rk3} holds with the following replacements: $e^{\vir}$ replaced by $\overline{\chi}_{-y}^{\vir}$, $\overline{\eta}(x^2)^{12}$ replaced by 
$$
\prod_{n=1}^{\infty} (1-x^{2n}y) (1-x^{2n}y^{-1}) (1-x^{2n})^{10},
$$
$\Theta_{A_2^\vee,(0,1)}(x)$ replaced by $\Theta_{A_2^\vee,(0,1)}(x,y)$, and $Z_{\pm}(x)$ replaced by $Z_{\pm}(x,y)$ which are the solutions to the following quadratic equation in $\zeta$
\begin{align*}
\zeta^2 - (Z(x,y)^2+3Z(x,y)Z(x,1)) \, \zeta +Z(x,y)+3Z(x,1) = 0,
\end{align*}
where  $Z(x,y) := \frac{\Theta_{A^\vee_2, (0,0)}(x,y)}{\Theta_{A^\vee_2, (0,1)}(x,y)}$. Note that $\overline{\eta}(x^6)^3$ does not get replaced.
\end{conjecture}

Specialising Conjectures \ref{conj:chiyvir:rk2}, \ref{conj:chiyvir:rk3} to $y=1$ yields Conjectures \ref{conj:evir:rk2}, \ref{conj:evir:rk3} respectively.

For K3 surfaces, by using deformation equivalence as in Section \ref{sec:eulervir}, Conjectures \ref{conj:chiyvir:rk2} and \ref{conj:chiyvir:rk3} are reduced to the calculation of $\chi_y$-genera of Hilbert schemes of points carried out by the first named author and W.~Soergel \cite{GS}. Furthermore, we verified Conjecture \ref{conj:chiyvir:rk2} for most surfaces listed in Section \ref{sec:eulervir:rk2}, and several values of $c_1$, but up to a lower virtual dimension than in the case of virtual Euler characteristics. More precisely, for virtual $\chi_y$-genera, the upper bound for the virtual dimension is usually between $5$ and $25$. See \cite[Sect.~7]{GK1} for the precise list of verifications.
Similarly, we verified Conjecture \ref{conj:chiyvir:rk3} for several of the surfaces listed in Section \ref{sec:eulervir:rk3}, for certain values of $c_1$, with upper bound on the virtual dimension between $2$ and $10$ (depending on the case). See \cite[Sect.~2.4]{GK3} for the precise list of verifications. The method we use for these verifications is discussed in Section \ref{sec:struc}.

\subsection{Application: $K$-theoretic S-duality} \label{sec:KVW}

Recently, D.~Maulik and Thomas \cite{MT} considered refinements of Vafa-Witten theory, in particular the $K$-theoretic Vafa-Witten invariants of a smooth projective surface $S$ worked out in \cite{Tho}. Let $N^\perp:=N_S^H(\rho,L,c_2)$ be a moduli space of stable Higgs pairs on a smooth polarized surface $(S,H)$. Consider
\begin{equation*} 
\chi(N^\perp,\O^{\vir}_{N^\perp}) := \chi( R\Gamma(N^\perp, \O_{N^\perp}^{\vir})),
\end{equation*}
viewed as a graded character.
As we already mentioned in Section \ref{sec:VW}, Vafa-Witten invariants of $S$ can be seen as reduced Donaldson-Thomas invariants counting 2-dimensional sheaves on $X = \mathrm{Tot}(K_S)$ \cite{GSY2}. N.~Nekrasov and A.~Okounkov \cite{NO} showed that in Donaldson-Thomas theory it is natural to replace the virtual structure sheaf $\O_{N^\perp}^{\vir}$ by its twisted version
$$
\widehat{\O}_{N^\perp}^{\vir} :=\O_{N^\perp}^{\vir} \otimes  \sqrt{K_{N^\perp}^{\vir}},
$$
where $\sqrt{K_{N^\perp}^{\vir}}$ is a choice of square root of $K_{N^\perp}^{\vir} = \det(\Omega_{N^\perp}^{\vir})$. Over the fixed locus $(N^\perp)^{\C^*}$, this choice of square root exists and is canonical \cite[Prop.~2.6]{Tho}. The $K$-theoretic Vafa-Witten invariants are defined by \cite[(2.12), Prop.~2.13]{Tho}
\begin{equation*} 
\chi(N^\perp,\widehat{\O}^{\vir}_{N^\perp}) = \chi \Big((N^\perp)^{\C^*}, \frac{\O^{\vir}_{(N^\perp)^{\C^*}}}{\Lambda_{-1} (N^{\vir})^{\vee}} \otimes \sqrt{K_{N^\perp}^{\vir}} \Big|_{(N^\perp)^{\C^*}} \Big).
\end{equation*} 
Here we use the notation \cite[Sect.~4]{FG}
\begin{equation} \label{totalwedge}
\Lambda_y V := \sum_{i=0}^{\rk(V)} [\Lambda^i V] \, y^i \in K^0(Z)[y], \quad \Lambda_y (V-W) := \frac{\Lambda_y V}{\Lambda_y W} \in K^0(Z)[[y]]
\end{equation}
for any classes $V,W$ of locally free sheaves of finite rank in the Grothendieck group $K^0(Z)$ of locally free sheaves of finite rank on a finite type $\C$-scheme $Z$. We use that $\Lambda_y V$ is an invertible element in $K^0(Z)[[y]]$. Recall from Section \ref{sec:VW} that $\mathfrak{t}$ denotes a degree one character of the $\C^*$-scaling action on $N^{\perp}$. Furthermore, we define
$$
t := c_1^{\C^*}(\mathfrak{t}), \quad y=e^t.
$$ 
One can show that $\chi(N^\perp,\widehat{\O}^{\vir}_{N^\perp})$ is invariant under $y \leftrightarrow y^{-1}$ \cite[Prop.~2.27]{Tho}. We denote the generating series of $K$-theoretic Vafa-Witten invariants by
$$
\sfZ_{S,H,\rho,L}(q,y) \in \Q[y^{\frac{1}{2}},y^{-\frac{1}{2}}](\!(q)\!),
$$
which is defined as in \eqref{mathdefVW} with $\VW_S^H(\rho,L,c_2)$ replaced by $\chi(N^\perp,\widehat{\O}^{\vir}_{N^\perp})$, where $N^\perp:=N_S^H(\rho,L,c_2)$.

Recall that $(N^\perp)^{\C^*}$ contains the Gieseker-Maruyama moduli space $M:=M_S^H(\rho,L,c_2)$ as a component. Thomas showed that its contribution to $\chi(N^\perp,\widehat{\O}^{\vir}_{N^\perp})$ equals, up to sign, the normalized virtual $\chi_y$-genus of $M$, i.e.
$$
(-1)^{\vd(M)} \overline{\chi}_{-y}^{\vir}(M).
$$

Let $(S,H)$ be a smooth polarized surface satisfying $H_1(S,\Z) = 0$ and $p_g(S)>0$. Analogous to the case of virtual Euler characteristics, in \cite{GK3} we made conjectures for the monopole contribution to $\chi(N^\perp,\widehat{\O}^{\vir}_{N^\perp})$ for $\rho=2$ and $\rho=3$, which are obtained as follows from the unrefined case.
\begin{itemize}
\item For rank $\rho=2$: take line one of \eqref{(5.38)} and replace $4\Delta(q^2)$ by $$\frac{\phi_{-2,1}(q^2,y^2)\Delta(q^2)}{(y^{\frac{1}{2}} - y^{-\frac{1}{2}})^2} = (y^{\frac{1}{2}} + y^{-\frac{1}{2}})^2 q^2 \prod_{n=1}^{\infty}(1-q^{2n}y^2)^2 (1-q^{2n}y^{-2})^2 (1-q^{2n})^{20},$$ where $\phi_{-2,1}(q,y)$ is a weak Jacobi form of weight $-2$ and index $1$, replace $\theta_2(q)$ by $\theta_2(q,y)$, and replace $\theta_3(q)$ by $\theta_3(q,y)$.
\item For rank $\rho=3$: take \eqref{conj:mono:rk3} and replace $9\Delta(q^3)$ by $$\frac{\phi_{-2,1}(q^3,y^3)\Delta(q^3)}{(y^{\frac{1}{2}} - y^{-\frac{1}{2}})^2},$$ replace $\Theta_{A_2,(1,0)}(q^{\frac{1}{2}})$ by $\Theta_{A_2,(1,0)}(q^{\frac{1}{2}},y)$, replace $W_{\pm}(q^{\frac{1}{2}})$ by $W_{\pm}(q^{\frac{1}{2}},y)$, where $W_{\pm}(x,y)$ are the solutions in $\omega$ of 
\begin{align*}
\omega^2 - (W(x,y)^2+3W(x,y)W(x,1)) \, \omega +W(x,y)+3W(x,1) = 0,
\end{align*}
where  $W(x,y) := \frac{\Theta_{A_2, (0,0)}(x,y)}{\Theta_{A_2, (1,0)}(x,y)}$.
\end{itemize}

By \cite{Laa1}, Theorem \ref{thm:Laa} also holds for the $(1, \ldots, 1)$ contribution to $K$-theoretic Vafa-Witten invariants; the only modification needed is that the universal functions have coefficients in $\Q(y^{\frac{1}{2}})$ instead of $\Q$. Using this, Laarakker \cite{Laa1} verified directly that the first few terms of these two monopole conjectures are correct. More precisely, he checked the first 15 terms for $\rho=2$ and the first 11 terms for $\rho=3$. 

Based on our conjectural formulae, we found a $K$-theoretic S-duality transformation, which we conjecture to be true for any prime rank $\rho$.
\begin{theorem} \cite[Prop.~4.8, 4.10]{GK3}
Our conjectural formulae for $\sfZ_{S,H,\rho,c_1}(q,y)$ for $\rho=2$ and $\rho=3$ (given in Conjectures \ref{conj:chiyvir:rk2}, \ref{conj:chiyvir:rk3}, and this section) are the Fourier expansions of meromorphic functions $\sfZ_{S,H,\rho,c_1}(\tau,z)$ on $\mathfrak{H} \times \C$ satisfying\footnote{The meaning of our generating functions for non-algebraic $c_1$ is as described in Remark \ref{non-alg}.}
\begin{align*} 
\begin{split}
\sfZ_{S,H,\rho,c_1}(\tau,z)\Big|_{(\tau+1,z)} =& \, (-1)^{\rho \chi(\O_S)} e^{\frac{\pi i \rho}{12} K_S^2} e^{-\frac{\pi i (\rho-1)}{\rho} c_1^2} \sfZ_{S,H,\rho,c_1}(\tau,z), \\
\frac{\sfZ_{S,H,\rho,c_1}(\tau,z)}{(y^{\frac{1}{2}} - y^{-\frac{1}{2}})^{\chi(\O_S)}}\Big|_{(-1/\tau,z/\tau)} =& \, (-1)^{\rho \chi(\O_S)} \rho^{1-\frac{e(S)}{2}} i^{-\frac{K_S^2}{2}} \tau^{-5\chi(\O_S)+\frac{K_S^2}{2}} e^{\frac{2 \pi i z^2}{\tau}\Big(-\frac{\rho}{2} \chi(\O_S) - \frac{\rho(\rho^2-1)}{24} K_S^2 \Big)} \\
& \, \times \sum_{[a]} e^{\frac{2 \pi i}{\rho} a c_1} \frac{\sfZ_{S,H,\rho,a}(\tau,z)}{(y^{\frac{1}{2}} - y^{-\frac{1}{2}})^{\chi(\O_S)}},
\end{split}
\end{align*}
where the sum is over all $[a] \in H^2(S,\Z) / \rho H^2(S,\Z)$, $q = e^{2 \pi i \tau}$, and $y = e^{2 \pi i z}$.
\end{theorem}

We perform further checks of the $K$-theoretic S-duality transformation in \cite[Sect.~4]{GK3}, namely for $\rho=1$, and for K3 surfaces and arbitrary prime rank  $\rho$. On the physics side, refined BPS indices were recently studied by S.~Alexandrov, J.~Manschot, and B.~Pioline \cite{AMP}.

\subsection{Virtual elliptic genera} \label{sec:ellvir}

The virtual elliptic genus of a proper $\C$-scheme $Z$ with perfect obstruction theory is defined by \cite{FG}
\begin{align*}
Ell^\vir(Z) &:= y^{-\frac{\vd(Z)}{2}} \sum_{p \geq 0} (-y)^p \, \chi(Z,\cE(T_Z^{\vir}) \otimes \Lambda^p \Omega_{Z}^{\vir} \otimes \O_Z^{\vir}), \\
\cE(T_Z^{\vir}) &:= \bigotimes_{n=1}^{\infty} \Lambda_{-yq^n} \Omega_Z^{\vir} \otimes \Lambda_{-y^{-1}q^n} T_Z^{\vir} \otimes \Sym_{q^n} (T_Z^{\vir} \oplus \Omega_Z^{\vir}),
\end{align*}
where $\Lambda_{y} V$ was defined in \eqref{totalwedge} and $\Sym_y V=\Lambda_{-y}(-V)$. Virtual elliptic genus refines complex elliptic genus, which has an interesting history (cf.~\cite{Hir, Wit, Kri}) that we will not discuss. When $Z$ is smooth and $T_Z^{\vir} = T_Z$, we write $Ell^\vir(Z) = Ell(Z)$.

Just like \eqref{eulerrk1} describes the Euler characteristics of $S^{[n]}$ in terms of $e(S)$, one can express the elliptic genera of $S^{[n]}$ in terms of $Ell(S)$. This is achieved by a famous formula originating from string theory in work of Dijkgraaf, G.~Moore, E.~Verlinde, and H.~Verlinde \cite{DMVV} and proved by L.~Borisov and A.~Libgober \cite{BL1, BL2}. In order to describe the formula, we need the notion of a Borcherds lift. For a formal series
$$
f(q,y) = \sum_{m \geq 0, n \in \Z} c_{m,n} q^m y^n,
$$
and any $a \in \Z$, we define a Borcherds type lift by
\begin{align*}
\begin{split} \label{L_a}
\sfL_a(f) := \prod_{l>0,m\geq 0,n \in \Z} (1-p^{al} q^m y^n)^{c_{lm,n}}.
\end{split}
\end{align*}
We set $\mathsf{L}(f) := \mathsf{L}_1(f)$. R.~Borcherds original definition \cite{Bor}, for meromorphic functions $f : \mathfrak{H} \times \C \rightarrow \C$, is given in terms of Hecke operators. The above formal version suffices for our purposes.
Later in this section, we will also encounter Borcherds type lifts of
$$
f^{\ev}(q,y) := \sum_{m \geq 0, n \in \Z} c_{2m,n} q^{2m} y^{n}.
$$
In addition, we will allow $y$ to have half-integer powers. 

For any smooth projective surface $S$, the Dijkgraaf-Moore-Verlinde-Verlinde formula states\footnote{The original formula in \cite{DMVV} is stated for orbifold elliptic genera of symmetric products $S^{(n)} := S^{n} / \mathfrak{S}_n$, where $\mathfrak{S}_n$ denotes the symmetric group of degree $n$.}
\begin{equation*} 
\sum_{n=0}^{\infty} Ell(S^{[n]}) \, p^n = \frac{1}{\sfL(Ell(S))}.
\end{equation*}
When $S$ is a K3 surface, this formula is of particular interest. The elliptic genus of a K3 surface is given by
\begin{equation*} 
Ell(\mathrm{K3}) = 2 \phi_{0,1}(q,y),
\end{equation*}
where $\phi_{0,1}(q,y)$ is a weak Jacobi form of weight 0 and index 1. (Together with $\phi_{-2,1}(q,y)$, encountered in the previous section, $\phi_{0,1}(q,y)$ generates the ring of weak Jacobi forms of even weight and integer index as a free algebra over the ring of modular forms for $\mathrm{SL}(2,\Z)$.) Moreover, V.~Gritsenko and V.~Nikulin \cite{GN} proved
\begin{equation*} \label{GN}
\sfL(2 \phi_{0,1}(q,y)) = \frac{\chi_{10}(p,q,y)}{p \, \Delta(q) \, \phi_{-2,1}(q,y)},
\end{equation*}
where $\chi_{10}(p,q,y)$ is the Igusa cusp form of weight 10 (a genus 2 Siegel modular form). Taken together, one obtains
$$
\sum_{n=0}^{\infty} Ell(\mathrm{K3}^{[n]}) \, p^{n-1} =  \frac{\Delta(q) \, \phi_{-2,1}(q,y)}{\chi_{10}(p,q,y)}.
$$

We present a rank 2 analog of the DMVV formula, which involves Borcherds type lifts of quasi- and weak Jacobi forms build from the following Jacobi-Eisenstein series
\begin{align*}
G_{1,0}(q,y)&:=-\frac{1}{2} \frac{y + 1}{y - 1}+\sum_{n=1}^{\infty} \sum_{d | n} (y^d-y^{-d}) q^n, \\
G_{k,0}(q,y)&:=\Big(y\frac{\partial}{\partial y}\Big)^{k-1}G_{1,0}(q,y), \quad \forall k>1.
\end{align*}
We define
\begin{align*}
\phi_{0,\frac{k}{2}}(q,y):=\,&G_{k,0}(q,y) \phi_{-2,1}(q,y)^{\frac{k}{2}}, \quad \forall k\ne 2.
\end{align*}

\begin{conjecture} \cite[Conj.~1.1, 7.7]{GK2} \label{conj:ellvir:rk2}
Let $(S,H)$ be a smooth polarized surface satisfying $b_1(S) = 0$ and $p_g(S)>0$. Suppose $M:=M_S^H(2,c_1,c_2)$ contains no strictly Gieseker $H$-semistable sheaves. Then $Ell^{\vir}(M)$ equals the coefficient of $p^{\vd(M)}$ of 
\begin{align*}
4 \Bigg(\frac{1}{2} A^{\elg}(p,q,y) \Bigg)^{\chi(\O_S)}\Bigg(2  B^{\elg}(p,q,y) \Bigg)^{K_{S}^2}  \sum_{a \in H^2(S,\Z)} \SW(a)(-1)^{a c_1} \Bigg(\frac{B^{\elg}(-p,q,y) }{B^{\elg}(p,q,y) }\Bigg)^{a K_S},
\end{align*}
where 
\begin{align*}
A^{\elg}(p,q,y) &:=  \frac{1}{\sfL_2(\phi_{0,1})} = \Bigg( \frac{p^2 \, \Delta(q) \, \phi_{-2,1}(q,y)}{\chi_{10}(p^2,q,y)} \Bigg)^{\frac{1}{2}}, \\
B^{\elg}(p,q,y) &:=  \frac{\sfL_4(2\phi_{0,\frac{1}{2}} \phi_{0,\frac{3}{2}}) \sfL(-2 \phi_{0,\frac{1}{2}})}{\sfL_2\big(-2 \phi_{0,\frac{1}{2}}^{\ev}|_{(q^{\frac{1}{2}},y)}-\phi_{0,\frac{1}{2}}|_{(q^2,y^2)}+2\phi_{0,\frac{1}{2}}^2 \big) }.
\end{align*}
\end{conjecture}
Specializing Conjecture \ref{conj:ellvir:rk2} to $q=0$ yields Conjecture \ref{conj:chiyvir:rk2}.

As in Section \ref{sec:eulervir}, Conjecture \ref{conj:ellvir:rk2} holds for K3 surfaces by deformation equivalence and Borisov-Libgober's result. Consider the following list of surfaces: \\

\noindent blow-ups up K3 surfaces in one point, elliptic surfaces of type $E(3), E(4), E(5), E(6)$, double covers of $\PP^2$ branched along a smooth octic, double covers of $\PP^1 \times \PP^1$ branched along a smooth curve of bidegree $(6,6)$ or $(6,8)$, double covers of the Hirzebruch surface $\FF_1 \rightarrow \PP^1$ branched along a smooth connected curve in the complete linear system $|\O_{\FF_1}(6B+10F)|$ where $B$ is the class of the section satisfying $B^2=-1$ and $F$ is a fibre class, smooth quintics in $\PP^3$. \\

For each of the surfaces in this list, we verified Conjecture \ref{conj:ellvir:rk2} for certain values of $c_1$ (sometimes with restrictions on $H$) and up to a certain virtual dimension, usually with upper bound between $8$ and $20$, as detailed in \cite[Sect.~8]{GK2}. The method we use for these verifications is discussed in Section \ref{sec:struc}.

\subsection{Virtual cobordism classes}

Finally, we turn our attention to algebraic cobordism theory \cite{LM, LP}. Denote the algebraic cobordism ring over a point with rational coefficients by
$$
\Omega_* := \bigoplus_{d=0}^{\infty} \Omega_d(\pt) \otimes_\Z \Q.
$$
Then $\Omega_*$ is isomorphic to the polynomial ring freely generated by the cobordism classes of $\PP^d$ for all $d \geq 0$. The graded piece $\Omega_d(\pt) \otimes_\Z \Q$ has a basis
$$
v^I := v_1^{i_1} \cdots v_d^{i_d}, \ \mathrm{where} \ I = (i_1, \ldots, i_d) \in \Z_{\geq 0}^d \ \mathrm{and} \ |I|=\sum k i_k = d.
$$
Concretely, the cobordism class $[Z]$ of a $d$-dimensional smooth projective variety $Z$ is 
$$
[Z] = \int_Z \prod_{i=1}^{d}\big(1+ \sum_{k=1}^{\infty} x_i^k v_k\big),
$$
where $x_1, \ldots, x_d$ are the Chern roots of $T_Z$. It follows that the class $[Z]$ is determined by the collection of all possible Chern numbers of $Z$ (i.e.~all possible intersection numbers obtained by capping monomials in Chern classes of $T_Z$ with $[Z]$). 

The cobordism classes of Hilbert schemes of points on surfaces were studied in \cite{EGL}. In loc.~cit., it is shown that there exist two universal functions $A, B \in 1+\Q[v_1,v_2,\ldots][\![p]\!]$ such that 
$$
\sum_{n=0}^{\infty} [S^{[n]}] \, p^{n} = A^{\chi(\O_S)} B^{K_S^2},
$$
for any smooth projective surface $S$. Consequently, $A^2$ is the generating function of cobordism classes of $\mathrm{K3}^{[n]}$. We now present a conjectural rank 2 analog of this formula.

Let $Z$ be a projective $\C$-scheme with a perfect obstruction theory. J.~Shen \cite{She} constructed a virtual cobordism class 
$$
[Z]^{\vir}_{\Omega_*} \in \Omega_{\vd}(Z),
$$
where $\vd = \rk T_Z^{\vir}$ (see also \cite{CFK} and \cite{LS} in the context of dg-manifolds and derived schemes). Denote by $\pi : Z \rightarrow \pt$ projection to a point. Shen proved that $\pi_* [Z]^{\vir}_{\Omega_*}$ is determined by the collection of virtual Chern numbers of $Z$ (i.e.~all possible intersection numbers obtained by capping monomials in Chern classes of $T_Z^{\vir}$ with $[Z]^{\vir}$). More precisely, let $T_Z^\vir = [E_0 \rightarrow E_1]$ be a resolution by vector bundles and denote the Chern roots of $E_0$ by $x_1, \ldots, x_n$ and the Chern roots of $E_1$ by $u_1, \ldots, u_m$. Then
\begin{equation} \label{vircobChern}
\pi_* [Z]^{\vir}_{\Omega_*} = \int_{[Z]^\vir} \frac{\prod_{i=1}^{n}\big(1+ \sum_{k=1}^{\infty} x_i^k v_k\big)}{\prod_{j=1}^{m}\big(1+ \sum_{k=1}^{\infty} u_j^k v_k\big)}.
\end{equation}
\begin{conjecture}  \cite[Conj.~1.2, 7.7]{GK2} \label{conj:vircob:rk2}
There exists a power series $B^{\cob}(p,\bfv) \in 1+\Q[v_1,v_2,\ldots][\![p]\!]$ with the following property. Let $(S,H)$ be a smooth polarized surface satisfying $b_1(S) = 0$ and $p_g(S)>0$. Suppose $M:=M_S^H(2,c_1,c_2)$ contains no strictly Gieseker $H$-semistable sheaves. Then $\pi_*[M]^{\vir}_{\Omega_*}$ equals the coefficient of $p^{\vd(M)}$ of 
\begin{align*}
4 \Bigg(\frac{1}{2} A^{\cob}(p,\bfv) \Bigg)^{\chi(\O_S)}\Bigg(2  B^{\cob}(p,\bfv) \Bigg)^{K_{S}^2}  \sum_{a \in H^2(S,\Z)} \SW(a)(-1)^{a c_1} \Bigg(\frac{B^{\cob}(-p,\bfv) }{B^{\cob}(p,\bfv) }\Bigg)^{a K_S},
\end{align*}
where 
\begin{align*}
A^{\cob}(p,\bfv) := \Big( \sum_{n=0}^{\infty} [\mathrm{K3}^{[n]}] \, p^{2n} \Big)^{\frac{1}{2}}.
\end{align*}
\end{conjecture}

By the virtual Hirzebruch-Riemann-Roch theorem of \cite{CFK, FG}, the virtual elliptic genera $Ell^{\vir}(M)$ of Conjecture \ref{conj:ellvir:rk2} can be expressed in terms of $q,y$ and virtual Chern numbers of $M$. As such, the universal functions $A^{\cob}(p,\bfv)$ and $B^{\cob}(p,\bfv)$ in Conjecture \ref{conj:vircob:rk2} determine the universal functions of Conjecture \ref{conj:ellvir:rk2}. Since we have no explicit formulae for $A^{\cob}(p,\bfv)$ and $B^{\cob}(p,\bfv)$, Conjecture \ref{conj:vircob:rk2} does not imply Conjecture \ref{conj:ellvir:rk2}. 

The universal function $A^{\cob}(p,\bfv)$ is determined modulo $p^{16}$ by calculations in \cite{EGL}. Assuming Conjecture \ref{conj:vircob:rk2} holds for the blow-up of an elliptic K3 surface and certain values of $H,c_1$, we determined  $B^{\cob}(p,\bfv)$ modulo $p^{14}$ and for $v_6=v_7 =\ldots =0$. The first few coefficients are
\begin{align*}
\frac{1}{B^{\cob}(p,\bfv)}=1 &+ 2v_1p - 16v_3p^3+ 4(v_1^4-3v_2v_1^2 + v_3v_1)p^4 \\ 
&+ 4(v_1^5- 6v_1^3v_2-12v_1^2v_3+9v_1v_2^2 + 22 v_2v_3   + 38v_5 )p^5+O(p^6).
\end{align*}
Conjecture \ref{conj:vircob:rk2} is verified in the same cases, and usually up to the same virtual dimension, as in Section \ref{sec:ellvir}. The method for the verifications is discussed in Section \ref{sec:struc}.

\begin{remark}
Remarkably, for any example of a non-zero virtual Chern number 
$$
\int_{[M]^{\vir}}c_{i_1}(T_M^{\vir})\cdots c_{i_k}(T_M^{\vir})
$$ 
that we calculated, we found an interesting positivity result. When $K_S^2>0$ and $c_2(S) > 0$, the virtual Chern number appears to have sign $(-1)^{\vd(M)}$.
This is similar to \cite[Rem.~5.5]{EGL}, where it is observed that all Chern numbers of $S^{[n]}$ are polynomials in $K_S^2$ and $c_2(S)$ with positive coefficients at least for $n\le 7$. 
\end{remark}

\section{Virtual Verlinde numbers} \label{sec:verlinde}

Let $C$ be a smooth projective curve of genus $g \geq 2$ and denote by $M$ the moduli space of rank 2 semistable vector bundles $E$ on $C$ with $\det E \cong \O_C$. The Picard group of $M$ is generated by the so-called determinant line bundle $\mathcal{L}$. The Verlinde formula, originating from conformal field theory \cite{Ver}, is the following remarkable expression
\begin{equation} \label{Verlindeorigin}
\dim H^0(M, \mathcal{L}^{\otimes r}) = \Big( \frac{r+2}{2} \Big)^{g-1} \sum_{j=1}^{r+1} \sin\Big( \frac{\pi j}{r+2} \Big)^{2-2g}, \quad \forall r \in \Z_{\geq 0}.
\end{equation}
We will not survey the rich literature on the Verlinde formula (see the introduction to \cite{GKW} for some references). In this section, we study analogs of the Verlinde formula for Gieseker-Maruyama moduli spaces on smooth projective surfaces.

\subsection{Hilbert schemes} \label{sec:verlinde:rk1}

Let $S$ be a smooth projective surface. The analog of the Verlinde formula for $S^{[n]}$ was studied in \cite{EGL}. We first describe the Picard group of $S^{[n]}$.
Any line bundle $L$ on $S$ induces a line bundle $L_n$ on the symmetric product $S^{(n)} := S^{n} / \mathfrak{S}_n$ by $\mathfrak{S}_n$-equivariant push-forward of $L \boxtimes \cdots \boxtimes L$ along the morphism $S^{n} \rightarrow S^{(n)}$. The pull-back of $L_n$ along the Hilbert-Chow morphism $S^{[n]} \rightarrow S^{(n)}$ is denoted by $\mu(L)$. Furthermore, consider
$
E:=\det \O_S^{[n]}.
$
The line bundles $\mu(L)$ and $E$ generate the Picard group of $S^{[n]}$. We consider the Verlinde numbers
$$
\chi(S^{[n]}, \mu(L) \otimes E^{\otimes r}).
$$
\begin{theorem}[Ellingsrud-G\"ottsche-Lehn] \label{EGLthm}
For any $r \in \Z$, there exist $g_r, f_r, A_r, B_r \in \Q[[w]]$ with the following properties. For any smooth projective surface $S$ and $L \in \Pic(S)$, we have
$$
\sum_{n=0}^{\infty} w^n \, \chi(S^{[n]}, \mu(L) \otimes E^{\otimes r}) = g_r^{\chi(L)} f_r^{\frac{1}{2} \chi(\O_S)} A_r^{L K_S} B_r^{K_S^2}.
$$
Moreover
\begin{align*}
g_r(w) = 1+v, \quad f_r(w) = (1+v)^{r^2}(1+ r^2 v)^{-1},
\end{align*}
where 
$w = v(1+v)^{r^2-1}.$
\end{theorem}

In \cite{EGL}, it is shown that $A_r=B_r = 1$ for $r = 0,\pm 1$. Using Serre duality and the (conjectural) Segre-Verlinde correspondence, discussed in Sections \ref{sec:SDvir} and \ref{sec:SVcorr} respectively, Marian-Oprea-Pandharipande \cite{MOP3} determined explicit formulae for $A_r, B_r$ for $r = \pm 2, \pm 3$.
Their calculations led to the following conjecture.
\begin{conjecture}[Marian-Oprea-Pandharipande] \label{MOPconj1}
$A_r$ and $B_r$ are algebraic functions for all $r$.
\end{conjecture}

\subsection{Rank 2} \label{sec:verlinde:rk2}

Let $(S,H)$ be a smooth polarized surface satisfying $b_1(S) = 0$ and let $M:=M_S^H(2,c_1,c_2)$. As usual, we assume $M$ does not contain strictly Gieseker $H$-semistable sheaves. 
Suppose a universal sheaf $\EE$ on $S \times M$ exists. Using the slant product 
$$
/ :  H^p(S \times M,\Q) \times H_q(S,\Q) \rightarrow H^{p-q}(M,\Q)
$$ 
and Poincar\'e duality on $S$, we define the $\mu$-insertion 
\begin{equation} \label{muinsert}
\mu(\alpha) :=   \big(c_2(\EE) - \frac{1}{4} c_1(\EE)^2 \big) / \mathrm{PD}(\alpha) \in H^*(M,\Q),
\end{equation}
for any $\alpha \in H^*(S,\Q)$. 

\begin{remark}
Although in general $\E$ only exists \'etale locally on $S \times M$, we can write
$$
c_2(\EE) - \frac{1}{4} c_1(\EE)^2 = -\frac{1}{4} \ch_2(\EE \otimes \EE \otimes \det(\EE)^*),
$$
where $\EE \otimes \EE \otimes \det(\EE)^*$ always exists globally on $S \times M$ (essentially because it is invariant under replacing $\EE$ by $\EE \otimes \mathcal{L}$ for any line bundle $\mathcal{L}$, so it glues from \'etale local patches). Hence $\mu(\alpha)$ is defined without assuming the existence of a universal sheaf $\E$ on $S \times M$.
\end{remark}

Let $L \in \Pic(S)$ be such that $c_1(L)c_1$ is even. Then there exists a line bundle $\mu(L)$ on $M$ such that its first Chern class is \eqref{muinsert} with $\alpha=c_1(L)$ \cite[Ch.~8]{HL}. The line bundle $\mu(L)$ is called a Donaldson line bundle. We first turn our attention to\footnote{When the Donaldson line bundle does not exist, we define $\chi(M,\mu(L) \otimes \O_M^{\vir})$ by the virtual Hirzebruch-Riemann-Roch formula: $\int_{[M]^{\vir}} e^{\mu(c_1(L))} \, \td(T_M^{\vir})$. Similarly for $\chi_{y}^{\vir}(M,\mu(L))$ below.} 
$$
\chi(M,\mu(L) \otimes \O^{\vir}_M ),
$$
which can be seen as a virtual Verlinde number and is also known as a $K$-theoretic Donaldson invariant \cite{GNY2}. The wall crossing behaviour of $K$-theoretic Donaldson invariants for toric surfaces was determined in \cite{GNY2}. The $K$-theoretic Donaldson invariants of rational surfaces, and their relationship to strange duality, were studied by the first-named author and Y.~Yuan  \cite{GY, Got5}. 

We are interested in the case $(S,H)$ is a smooth polarized surface satisfying $b_1(S) = 0$ and $p_g(S)>0$. Let $L\in \Pic(S)$ and suppose $M:=M_S^H(2,c_1,c_2)$ does not contain strictly Gieseker $H$-semistable sheaves. In \cite[Conj.~1.1]{GKW}, together with R.A.~Williams, we conjectured that $\chi(M,\mu(L) \otimes \O^{\vir}_M )$ is given by the coefficient of $x^{\vd(M)}$ of  
\begin{equation} \label{conj:Verlindevir:rk2}
\frac{2^{2-\chi(\O_S)+K_S^2}}{(1-x^2)^{\frac{(L-K_S)^2}{2}+\chi(\O_S)}} \sum_{a \in H^2(S,\Z)} \SW(a) \, (-1)^{ac_1} \, \left(\frac{1+x}{1-x}\right)^{\left(\frac{K_S}{2}-a\right)(L-K_S)}.
\end{equation}
There are several directions into which \eqref{conj:Verlindevir:rk2} can be generalized. In Section \ref{sec:verlinde:rkgeneral}, we present a generalization to more general line bundles on $M$ and higher rank Gieseker-Maruyama moduli spaces. Another interesting generalization concerns ``virtual $\chi_y$-genus valued in a Donaldson line bundle''
$$
\chi_{y}^{\vir}(M,\mu(L)) :=  \sum_{p} y^p \, \chi(M, \mu(L) \otimes \Lambda^p \Omega^{\vir}_M \otimes \O_M^{\vir})
$$
and its normalized version $\overline{\chi}_{-y}^{\vir}(M,\mu(L)) := y^{-\frac{\vd(M)}{2}}\chi_{-y}^{\vir}(M,\mu(L))$. Together with Williams, we conjectured the following formula.
\begin{conjecture} \cite[Conj.~1.2]{GKW} \label{conj:Verlindechiyvir:rk2}
Let $(S,H)$ be a smooth polarized surface satisfying $b_1(S) = 0$, $p_g(S)>0$, and let $L \in \Pic(S)$. Suppose $M:=M_S^H(2,c_1,c_2)$ contains no strictly Gieseker $H$-semistable sheaves. Then $\overline{\chi}_{-y}^{\vir}(M,\mu(L))$ equals the coefficient of $x^{\vd(M)}$ of 
{\scriptsize{
\begin{align*}
&4\left( \frac{1}{2} \prod_{n=1}^{\infty} \frac{1}{(1-x^{2n})^{10}(1-x^{2n}y)(1-x^{2n}y^{-1})}\right)^{\chi(\O_S)}\left(\frac{ 2 \overline \eta(x^4)^2}{\theta_3(x,y^{\frac{1}{2}})}\right)^{K_S^2} \left(\prod_{n=1}^{\infty} \left(\frac{(1-x^{2n})^2}{(1-x^{2n}y)(1-x^{2n}y^{-1})}\right)^{n^2}\right)^{\frac{L^2}{2}} \times \\
&\left(\prod_{n=1}^{\infty}\left( \frac{1-x^{2n}y^{-1}}{1-x^{2n}y}\right)^n\right)^{LK_S} \!\!\!\!\! \sum_{a \in H^2(S,\Z)} (-1)^{ac_1} \, \SW(a) \,   \left(\frac{\theta_3(x,y^{\frac{1}{2}})}{\theta_3(-x,y^{\frac{1}{2}})}\right)^{aK_S} \Bigg(\prod_{n=1}^{\infty}\left(\frac{(1-x^{2n-1} y^{\frac{1}{2}})(1+x^{2n-1}y^{-\frac{1}{2}})}{(1-x^{2n-1}y^{-\frac{1}{2}})(1+x^{2n-1} y^{\frac{1}{2}})}\right)^{2n-1}\Bigg)^{\frac{L(K_S-2a)}{2}}.
\end{align*}
}}
\end{conjecture}

Similar to the discussion in Sections \ref{sec:eulervir} and \ref{sec:refine}, for K3 surfaces one can reduce Conjecture \ref{conj:Verlindechiyvir:rk2} to the calculation of $\chi_{y}^{\vir}(S^{[n]},\mu(L))$ \cite{Got6}. Moreover, the first-named author derived a formula for these numbers \cite{Got6} (and in fact, more generally, for elliptic genera of Hilbert schemes of points with values in a Donaldson line bundle). This establishes the case of K3 surfaces. Furthermore, consider the following list of surfaces: \\

\noindent K3 surfaces blown-up in at most two points, elliptic surfaces of type $E(3), E(4), E(5)$, blow-ups of an elliptic surface of type $E(3)$ in one point, double covers of $\PP^2$ branched along a smooth octic, blow-ups of the previous surfaces in one point, smooth quintics in $\PP^3$, blow-ups of the previous surfaces in one point. \\

As before, in each case we verified Conjecture \ref{conj:Verlindechiyvir:rk2} for certain values of $c_1$ (sometimes with conditions on $H$) and up to a certain virtual dimension as detailed in \cite[Sect.~2.5]{GKW}. Our method for these verifications is described in Section \ref{sec:struc}.

\subsection{Application: Verlinde formula for Higgs pairs}

The original Verlinde formula \eqref{Verlindeorigin} was recently upgraded to the moduli space of semistable Higgs pairs on a smooth projective curve by D.~Halpern-Leistner \cite{H-L} and J.E.~Andersen, S.~Gukov, and Du Pei \cite{AGP}. We now discuss an extension of Conjecture \ref{conj:Verlindechiyvir:rk2} to Higgs pairs.  

Let $(S,H)$ be a smooth polarized surface satisfying $H_1(S,\Z) = 0$, $p_g(S)>0$, and let $L \in \Pic(S)$. Consider the moduli space of rank 2 Higgs pairs $N^\perp := N_S^H(2,c_1,c_2)$ introduced in Section \ref{sec:VW}. As before, we assume $N^\perp$ does not contain strictly Gieseker $H$-semistable Higgs pairs. In \cite{GKW}, we studied the Verlinde numbers
\begin{equation} \label{VerlindeHiggs}
\chi(N^\perp, \mu(L) \otimes \widehat{\O}^{\vir}_{N^\perp}),
\end{equation}
which we define by a $K$-theoretic virtual formula similar to Section \ref{sec:KVW} (see \cite[Sect.~1.2]{GKW} for details). The instanton contribution to this invariant is 
$$
(-1)^{\vd(M)} \overline{\chi}_{-y}^{\vir}(M,\mu(L)),
$$
where $M := M_S^H(2,c_1,c_2)$ is the Gieseker-Maruyama moduli space and $y$ relates to the $\C^*$-equivariant parameter $t = c_1^{\C^*}(\mathfrak{t})$ via $y = e^t$. This contribution is determined by Conjecture \ref{conj:Verlindechiyvir:rk2}. In \cite[Conj.~1.3]{GKW}, we give a conjectural formula for the monopole contribution to \eqref{VerlindeHiggs} in a very similar shape to the formula of Conjecture \ref{conj:Verlindechiyvir:rk2}. We prove this monopole formula for K3 surfaces. More generally, we prove an analog of Theorem \ref{thm:Laa} for the monopole contribution to \eqref{VerlindeHiggs}. As in Laarakker's calculations, this allows us to show that the first 15 terms of our prediction for the monopole contribution is correct.

Our conjectural formula for $\chi(N^\perp, \mu(L) \otimes \widehat{\O}^{\vir}_{N^\perp})$ interpolates between $K$-theoretic Donaldson invariants and $K$-theoretic Vafa-Witten invariants:
\begin{itemize}
\item \textbf{$K$-theoretic Donaldson invariants.} Replacing $x$ by $x y^{\frac{1}{2}}$ in the formula of Conjecture \ref{conj:Verlindechiyvir:rk2} gives an expression for $\chi_{-y}^{\vir}(M,\mu(L))$. Setting $y=0$ yields the formula for rank 2 $K$-theoretic Donaldson invariants \eqref{conj:Verlindevir:rk2}.
\item \textbf{$K$-theoretic Vafa-Witten invariants.} Taking $L=\O_S$ in Conjecture \ref{conj:Verlindechiyvir:rk2} and its monopole analog \cite[Conj.~1.3]{GKW}, we obtain the conjectural formula for the rank 2 $K$-theoretic Vafa-Witten invariants of Conjecture \ref{conj:chiyvir:rk2} and Section  \ref{sec:KVW}. 
\end{itemize}
In \cite[App.~B]{GK1}, the first-named author and Nakajima conjectured a formula interpolating between the Donaldson invariants and virtual Euler characteristics of $M := M_S^H(2,c_1,c_2)$. Conjecture \ref{conj:Verlindechiyvir:rk2} implies this formula (\cite[Prop.~4.8]{GKW}).

\subsection{Arbitrary rank} \label{sec:verlinde:rkgeneral}

We want to generalize \eqref{conj:Verlindevir:rk2} to more general line bundles and higher rank Gieseker-Maruyama moduli spaces. Let $(S,H)$ be a smooth polarized surface satisfying $b_1(S)=0$ and consider $M:=M_S^H(\rho,c_1,c_2)$  for any $\rho>0$. Assume $M$ does not contain strictly Gieseker $H$-semistable sheaves. We describe the analogs of the line bundles $\mu(L) \otimes E^{\otimes r}$ on  $S^{[n]}$ (cf.~\cite[Ch.~8]{HL}). 
We first suppose there exists a universal sheaf $\E$ on $S \times M$, but we point out in Remark \ref{univdrop} below how to get rid of this assumption. Consider 
\begin{equation} \label{lambdaE}
\lambda_{\EE} : K^0(S) \rightarrow \Pic(M), \quad \alpha \mapsto \det \big( \pi_{M!} \big( \pi_S^* \alpha \cdot [\E] \big) \big)^{-1},
\end{equation}
where $\pi_{M!} = \sum_i (-1)^i R^i \pi_{M*}$.
We fix a class $c \in K(S)_{\mathrm{num}}$ in the numerical Grothendieck group of $S$ satisfying $\rk(c) = \rho$, $c_1(c) = c_1$, $c_2(c) = c_2$. Restricting $\lambda_{\EE}$ to
\begin{equation} \label{Kc}
K_c := \{ v \in K^0(S) \, : \, \chi(S,c \otimes v) = 0\},
\end{equation}
the map $\lambda_{\EE} =: \lambda$ becomes independent of the choice of universal sheaf $\EE$ \cite[Ch.~8]{HL}.

Let $r \in \Z$, $L \in \Pic(S) \otimes \Q$ such that $\mathcal{L}:=L \otimes \det(c)^{-\frac{r}{\rho}} \in \Pic(S)$ and $\rho$ divides 
$\mathcal{L}c_1 + r \big( \frac{1}{2} c_1(c_1-K_S) - c_2 \big)$. 
Take a class $v \in K^0(S)$ satisfying:
\begin{itemize}
\item $\rk(v) = r$ and $c_1(v) = \mathcal{L}$, 
\item $c_2(v)=\frac{1}{2} \mathcal{L}( \mathcal{L}-K_S) +r\chi(\O_S) +\frac{1}{\rho}\mathcal{L}c_1 +\frac{r}{\rho}\big(\frac{1}{2}c_1(c_1-K_S)-c_2\big)$.
\end{itemize}
The second condition is equivalent to $v \in K_c \subset K^0(S)$. We define
\begin{equation} \label{muLEr}
\mu(L) \otimes E^{\otimes r} := \lambda(v).
\end{equation}

\begin{remark}
For $\rho = 1$ and $c_1=0$, \eqref{muLEr} coincides with the definition of $\mu(L) \otimes E^{\otimes r}$ on $M_S^H(1,0,n) \cong S^{[n]}$ introduced in Section \ref{sec:verlinde:rk1} (by \cite[Rem.~5.3(2)]{Got6}). For $r=0$, \eqref{muLEr} coincides with the definition of $\mu(L)$ in Section \ref{sec:verlinde:rk2} (by \cite[Rem.~5.3(1)]{Got6}).
\end{remark}

\begin{remark} \label{univdrop}
Without assuming the existence of a universal sheaf $\E$ on $S \times M$, there still exists a homomorphism
$
\lambda : K_c \longrightarrow \Pic(M) 
$
such that for any morphism $\phi : B \rightarrow M$ and any $B$-flat family of coherent sheaves $\F$ on $S \times B$, we have $\phi^* \lambda(v) = \lambda_{\F}(v)$ for all $v \in K^0(S)$. Here $K_c$, $\lambda_{\F}$ are defined in \eqref{Kc}, \eqref{lambdaE} (with $M$ replaced by $B$ and $\E$ by $\F$). See \cite[Ch.~8]{HL}, \cite[Sect.~1.1]{GNY2}. Given this $\lambda$, one then defines $\mu(L) \otimes E^{\otimes r}$ by \eqref{muLEr}.
\end{remark}

\begin{conjecture} \cite{GK4} \label{conj:Verlindevir}
Let $\rho > 0$ and $r \in \Z$. There exist  $G_r$, $F_{r} \in \C[[w]]$, $A_r$, $B_r$, $A_{i,r}$, $B_{ij,r} \in \C[[w^{\frac{1}{2}}]]$, for all $1 \leq i \leq j \leq \rho-1$, with the following property.\footnote{These universal functions depend on $\rho$ and $r$. We suppress the dependence on $\rho$.} Let $(S,H)$ be a smooth polarized surface satisfying $b_1(S) = 0$,  $p_g(S)>0$, and let $L \in \Pic(S)$. Suppose $M:=M_S^H(\rho,c_1,c_2)$ contains no strictly Gieseker $H$-semistable sheaves. Then $\chi(M, \mu(L) \otimes E^{\otimes r} \otimes \O_M^{\vir})$
equals the coefficient of $w^{\frac{1}{2}\vd(M)}$ of
\begin{align*} 
\rho^{2 - \chi(\O_S)+K_S^2} \, G_{r}^{\chi(L)} F_{r}^{\frac{1}{2} \chi(\O_S)} A_{r}^{L K_S} B_{r}^{K_S^2} \sum_{(a_1, \ldots, a_{\rho-1})} \prod_{i=1}^{\rho-1} \epsilon_{\rho}^{i a_i c_1} \, \SW(a_i) \, A_{i,r}^{a_i L} \prod_{1 \leq i \leq j \leq \rho-1} B_{ij,r}^{a_i a_j},
\end{align*}
where the sum is over all $(a_1, \ldots, a_{\rho-1}) \in H^2(S,\Z)^{\rho -1}$ and $\epsilon_\rho :=e^{2 \pi \sqrt{-1} / \rho}$. Furthermore, $A_r$, $B_r$, $A_{i,r}$, $B_{ij,r}$ are algebraic functions for all $r,i,j$.
\end{conjecture}

When $S$ is a K3 surface, it is shown in \cite[Prop.~1.10]{GNY2} that deformation equivalence together with a result of A.~Fujiki can be used to express the Verlinde numbers of $M$ in terms of those of $S^{[\frac{1}{2}\vd(M)]}$. The latter are determined by Theorem \ref{EGLthm}. Hence Conjecture \ref{conj:Verlindevir} is true for K3 surfaces and 
\begin{align}
\begin{split} \label{fg}
G_r(w) &= g_{r/\rho}(w) = 1+v, \\
F_r(w) &= f_{r/\rho}(w) = (1+v)^{\frac{r^2}{\rho^2}}\Big(1+\frac{r^2}{\rho^2} v\Big)^{-1},
\end{split}
\end{align}
where $w = v(1+v)^{\frac{r^2}{\rho^2}-1}.$  

For $\rho=2$, $3$, $4$, and several values of $r$, we have explicit (conjectural) algebraic expressions for $A_r$, $B_r$, $A_{i,r}$, $B_{ij,r}$ \cite{GK4}. We present some examples of these in Section \ref{sec:alg}. Similar to previous sections, we verified Conjecture \ref{conj:Verlindevir} for $\rho=2$, $3$, $4$, and various values of $r$ for a certain list surfaces and up to certain virtual dimensions (using the strategy outlined in Section \ref{sec:struc}). The precise list of verifications can be found in \cite{GK4}.

\subsection{Virtual Serre duality} \label{sec:SDvir}

Applying virtual Serre duality \cite[Prop.~3.13]{FG}, to the Verlinde numbers of Sections \ref{sec:verlinde:rk1} and \ref{sec:verlinde:rkgeneral} gives
\begin{align*}
\chi(M, \mu(L) \otimes E^{\otimes r} \otimes \O_M^{\vir}) &= (-1)^{\vd(M)} \chi(M, \mu(-L) \otimes E^{\otimes -r} \otimes K_{M}^{\vir} \otimes \O_M^{\vir}) \\
&= (-1)^{\vd(M)} \chi(M, \mu(-L+\rho K_S) \otimes E^{\otimes -r} \otimes \O_M^{\vir}),
\end{align*}
where $K_M^{\vir}:= \Lambda^{\vd(M)} \Omega_M^{\vir}$ and we use $c_1(T_M^{\vir}) = - \rho \, \mu(K_S)$ \cite[Prop.~8.3.1]{HL}. This puts constraints on the universal functions of Theorem \ref{EGLthm} and Conjecture \ref{conj:Verlindevir}. We already know that 
\begin{align*}
f_{-r/\rho} = f_{r/\rho}, \quad g_{-r/\rho} = g_{r/\rho},
\end{align*}
for all $r \in \Z$. In addition, for rank $\rho=1$, we have (see also \cite{EGL})
$$
A_r = \frac{B_{-r}}{B_r},
$$
for all $r \in \Z$. For any $\rho>0$, virtual Serre duality suggests the following relations.
\begin{conjecture} \cite{GK4} \label{conj:SDvir}
For any $\rho > 0$, we have
\begin{align*}
B_{-r}(w^{\frac{1}{2}})&=g_{r/\rho}(w)^{\binom{\rho}{2}}A_{r}(-w^{\frac{1}{2}})^\rho B_{r}(-w^{\frac{1}{2}}), \\
B_{ii,-r}(w^{\frac{1}{2}})&=A_{i,r}(-w^{\frac{1}{2}})^\rho B_{ii,r}(-w^{\frac{1}{2}}),\\
B_{ij,-r}(w^{\frac{1}{2}})&= B_{ij,r}(-w^{\frac{1}{2}}),
\end{align*}
for all $i=1, \ldots, \rho-1$ and $1 \leq i<j \leq \rho-1$.
\end{conjecture}
As a consequence of this conjecture, the universal functions $A_r, B_r, A_{i,r}, B_{ij,r}$ with $r<0$ are determined by the universal functions with $r>0$ (and vice versa).
In the cases where we have explicit (conjectural) algebraic expressions for $A_r, B_r, A_{i,r}$, $B_{ij,r}$, we show that they satisfy the equations of this conjecture (see also Section \ref{sec:alg}).

\section{Virtual Segre numbers} \label{sec:segre}

\subsection{Hilbert schemes (Marian-Oprea-Pandharipande)}

Let $S$ be a smooth projective surface. Similar to the tautological bundles $L^{[n]}$ in the introduction, we can consider $K$-theoretic tautological classes as follows. For any $\alpha \in K^0(S)$, we define
\begin{equation*}
\alpha^{[n]} := q_! (p^* \alpha) \in K^0(S^{[n]}),
\end{equation*}
where $p$ and $q$ are projections from the universal subscheme as in the introduction. In \cite{MOP3}, Marian-Oprea-Pandharipande prove the following remarkable theorem.
\begin{theorem}[Marian-Oprea-Pandharipande] \label{MOPthm}
For any $s \in \Z$, there exist $V_s$, $W_s$, $X_s$, $Y_s$, $Z_s \in \Q[[z]]$ with the following property. For any smooth projective surface $S$ and $\alpha \in K^0(S)$ of rank $s$, we have
$$
\sum_{n=0}^{\infty} z^n \int_{S^{[n]}} c(\alpha^{[n]}) = V_s^{c_2(\alpha)} W_s^{c_1(\alpha)^2} X_s^{\chi(\O_S)} Y_s^{c_1(\alpha) K_S} Z_s^{K_S^2}.
$$
Moreover
\begin{align*}
V_s(z) &= (1+(1-s)t)^{1-s} (1+(2-s)t)^s, \\ 
W_s(z) &= (1+(1-s)t)^{\frac{1}{2}s-1} (1+(2-s)t)^{\frac{1}{2}(1-s)} \\
X_s(z) &= (1+(1-s)t)^{\frac{1}{2} s^2-s} (1+(2-s)t)^{-\frac{1}{2}s^2+\frac{1}{2}}(1+(1-s)(2-s)t)^{-\frac{1}{2}},
\end{align*}
where 
$$z = t(1+(1-s)t)^{1-s}.$$
\end{theorem}

As we discussed in the introduction, Lehn's conjecture provides explicit formulae for $(V_{-1}W_{-1})$, $X_{-1}$,$Y_{-1}$,$Z_{-1}$, cf.~\eqref{lehnconj}.\footnote{Note that \cite{MOP3} use a different change of variables compared to \cite{Leh}. Hence the formulae in Theorem \ref{MOPthm} and \eqref{lehnconj} look different.}
Lehn's conjecture was established in \cite{MOP2} building on \cite{MOP1, Voi}. 
Furthermore, Marian-Oprea-Pandharipande proved closed formulae for $Y_s, Z_s$ 
for $s \in \{-2,-1,1,2\}$ in \cite{MOP3}  and gave a conjectural formula for $Y_0$ (note: $Z_0=1$ is trivial).
This led to the following conjecture.
\begin{conjecture}[Marian-Oprea-Pandharipande] \label{MOPconj2}
$Y_s$ and $Z_s$ are algebraic functions for all $s$.
\end{conjecture}

\subsection{Arbitrary rank} \label{sec:Segre:arbrk}

We are interested in virtual Segre numbers on Gieseker-Maruyama moduli spaces of any rank on any smooth polarized surface $(S,H)$ satisfying $b_1(S)=0$. This requires us to define the analog of the tautological classes $\alpha^{[n]}$. As before, we consider $M:=M_S^H(\rho,c_1,c_2)$ for any $\rho>0$. We assume $M$ does not contain strictly Gieseker $H$-semistable sheaves. For the moment we also assume there exists a universal sheaf $\E$ on $S \times M$. For any class $\alpha \in K^0(S)$, we define
$$
\ch(\alpha_M) :=  \ch(- \pi_{M !} ( \pi_S^* \alpha \cdot \E \cdot  \det(\E)^{-\frac{1}{\rho}})) \in A^*(M)_\Q,
$$
where $A^*(M)_{\Q}$ denotes the Chow ring with rational coefficients. When the root $\det(\E)^{-1/\rho}$ does not exist, the right hand side is defined by a formal application of the Grothendieck-Riemann-Roch formula. 
We note the following:
\begin{itemize}
\item For $c_1=0$, $M:=M_S^H(1,0,n) \cong S^{[n]}$ and $\ch_i(\alpha_M) = \ch_i(\alpha^{[n]})$ for all $i>0$.
\item $\ch(\alpha_M)$ is invariant upon replacing $\E$ by $\E \otimes \mathcal{L}$ for any line bundle $\mathcal{L}$ on $M$ (due to the factor $\det(\E)^{-\frac{1}{\rho}}$). Hence $\alpha_M$ is independent of the choice of universal sheaf.
\item After applying the Grothendieck-Riemann-Roch formula, the right hand side involves the expression $\ch(\E \otimes \det(\E) ^{-1/\rho})$ which can be rewritten as $\ch(\E^{\otimes \rho} \otimes \det(\E) ^{-1})^{1/\rho}$. The sheaf $\E^{\otimes \rho} \otimes \det(\E) ^{-1}$ always exists on $S \times M$ also when the universal sheaf $\E$ does not exist globally on $S \times M$. In this way, the insertion $\ch(\alpha_M)$ is defined without assuming the existence of a universal sheaf $\E$ on $S \times M$.
\end{itemize}

\begin{conjecture} \label{conj:Segrevir} \cite{GK4}
Let $\rho > 0$ and $s \in \Z$. There exist  $V_s$, $W_s$, $X_s \in \C[[z]]$, $Y_s$, $Z_s$, $Y_{i,s}$, $Z_{ij,s} \in \C[[z^{\frac{1}{2}}]]$, for all $1 \leq i \leq j \leq \rho-1$, with the following property.\footnote{These universal functions depend on $\rho$ and $s$. We suppress the dependence on $\rho$.} Let $(S,H)$ be a smooth polarized surface satisfying $b_1(S) = 0$ and  $p_g(S)>0$. Suppose $M:=M_S^H(\rho,c_1,c_2)$ contains no strictly Gieseker $H$-semistable sheaves. For any $\alpha \in K^0(S)$ such that $\rk(\alpha) = s$, the virtual Segre number $\int_{[M]^{\vir}} c( \alpha_M )$
equals the coefficient of $z^{\frac{1}{2} \vd(M)}$ of
\begin{align*}
\rho^{2 - \chi(\O_S)+K_S^2} \, V_s^{c_2(\alpha)} W_s^{c_1(\alpha)^2} X_s^{\chi(\O_S)} Y_{s}^{c_1(\alpha) K_S} Z_{s}^{K_S^2} \sum_{(a_1, \ldots, a_{\rho-1})} \prod_{i=1}^{\rho-1} \epsilon_{\rho}^{i a_i c_1} \, \SW(a_i) \, Y_{i,s}^{c_1(\alpha) a_i} \prod_{1 \leq i \leq j \leq \rho-1} Z_{ij,s}^{a_i a_j},
\end{align*}
where the sum is over all $(a_1, \ldots, a_{\rho-1}) \in H^2(S,\Z)^{\rho-1}$ and $\epsilon_\rho :=e^{2 \pi \sqrt{-1} / \rho}$. Moreover
\begin{align*}
V_s(z) &= \Big(1+\Big(1-\frac{s}{\rho}\Big)t\Big)^{1-s} \Big(1+\Big(2-\frac{s}{\rho}\Big)t\Big)^s \Big(1+\Big(1-\frac{s}{\rho}\Big)t\Big)^{\rho-1} , \\ 
W_s(z) &= \Big(1+\Big(1-\frac{s}{\rho}\Big)t\Big)^{\frac{1}{2}s-1} \Big(1+\Big(2-\frac{s}{\rho}\Big)t\Big)^{\frac{1}{2}(1-s)} \Big(1+\Big(1-\frac{s}{\rho}\Big)t\Big)^{\frac{1}{2} - \frac{1}{2} \rho}, \\
X_s(z) &=  \Big(1+\Big(1-\frac{s}{\rho}\Big)t\Big)^{\frac{1}{2} s^2-s} \Big(1+\Big(2-\frac{s}{\rho}\Big)t\Big)^{-\frac{1}{2}s^2+\frac{1}{2}} \Big(1+\Big(1-\frac{s}{\rho}\Big)\Big(2-\frac{s}{\rho}\Big)t\Big)^{-\frac{1}{2}} \\
&\quad\quad \times \Big(1+\Big(1-\frac{s}{\rho}\Big)t\Big)^{-\frac{(\rho-1)^2}{2\rho} s}, 
\end{align*}
where $$z = t \Big(1+\Big(1-\frac{s}{\rho}\Big) t\Big)^{1-\frac{s}{\rho}}.$$ Furthermore, $Y_s$, $Z_{s}$, $Y_{i,s}$, $Z_{ij,s}$ are algebraic functions for all $s,i,j$.
\end{conjecture}
For $\rho=2$, $3$, $4$, and various values of $s$, we have explicit (conjectural) algebraic expressions for $Y_s$, $Z_s$, $Y_{i,s}$, $Z_{ij,s}$ \cite{GK4}. We give some examples of these in Section \ref{sec:alg}. Similar to previous sections, we verified Conjecture \ref{conj:Segrevir} for $\rho=2$, $3$, $4$, and various values of $s$ for a certain list surfaces and up to certain virtual dimensions (using the strategy outlined in Section \ref{sec:struc}). The precise list of verifications can be found in \cite{GK4}.

\subsection{Virtual Segre-Verlinde correspondence} \label{sec:SVcorr}

In the rank 1 case, using the explicit expressions for the universal functions of Theorems \ref{EGLthm} and \ref{MOPthm}, one obtains
\begin{align*}
f_{r}(w) &= W_s(z)^{-4s} X_s(z)^2, \\
 g_{r}(w) &= V_s(z) W_s(z)^2, 
\end{align*}
where $s = 1+r$ and 
\begin{equation} \label{varchange1}
w = v(1+v)^{r^2-1}, \quad z = t(1+(1-s)t)^{1-s}, \quad v=t(1- r t)^{-1}.
\end{equation} 
Based on work of D.~Johnson \cite{Joh}, which was motivated by strange duality, Marian-Oprea-Pandharipande \cite{MOP3} formulated the following ``Segre-Verlinde correspondence''.\footnote{We slightly restated the formulation of \cite{MOP3} by connecting the variables $v,t$ via $v=t(1- r t)^{-1}$.}
\begin{conjecture}[Johnson, Marian-Oprea-Pandharipande] 
For any $r \in \Z$, $s=1+r$, and under the formal variable change \eqref{varchange1}, we have 
\begin{align*}
A_{r}(w) &= W_s(z)Y_s(z), \\  
B_{r}(w) &= Z_s(z).
\end{align*}
\end{conjecture}
In particular, this conjecture implies that Conjectures \ref{MOPconj1} and \ref{MOPconj2} are equivalent.

Similar to the rank 1 case, for any $\rho>0$ and $s \in \Z$, a direct calculation shows that the universal functions of Conjectures \ref{conj:Verlindevir} (equation \eqref{fg}) and \ref{conj:Segrevir} are related as follows
\begin{align*}
f_{r/\rho}(w) &= V_s(z)^{\frac{s}{\rho} (\rho^{\frac{1}{2}} - \rho^{-\frac{1}{2}})^2} W_s(z)^{-\frac{4s}{\rho}} X_s(z)^2, \\
g_{r/\rho}(w) &= V_s(z) W_s(z)^2,
\end{align*}
where $s = \rho+r$ and 
\begin{equation} \label{varchange2}
w = v(1+v)^{\frac{r^2}{\rho^2}-1}, \quad z = t \Big(1+\Big(1-\frac{s}{\rho}\Big) t\Big)^{1-\frac{s}{\rho}}, \quad v=t\Big(1-\frac{r}{\rho}t \Big)^{-1}.
\end{equation}
We present a ``virtual Segre-Verlinde correspondence'' for arbitrary rank $\rho$.
\begin{conjecture} \label{conj:SVcorrvir} \cite{GK4}
For any $\rho>0$, $r \in \Z$, $s=\rho+r$, and under the formal variable change \eqref{varchange2}, we have 
\begin{align*}
A_{r}(w^{\frac{1}{2}}) &= W_{s}(z) Y_{s}(z^{\frac{1}{2}}), \quad A_{i,r}(w^{\frac{1}{2}}) = Y_{i,s}(z^{\frac{1}{2}}),\\ 
B_{r}(w^{\frac{1}{2}}) &= Z_{s}(z^{\frac{1}{2}}), \quad \quad \quad \, B_{ij,r}(w^{\frac{1}{2}}) = Z_{ij,s}(z^{\frac{1}{2}}),
\end{align*}
for all $1 \leq i \leq j \leq \rho-1$.\footnote{The series $A_{r}, B_{r}, \ldots$ and $Y_{s}, Z_{s}, \ldots$ depend on $w^{\frac{1}{2}}$ and $z^{\frac{1}{2}}$, so strictly speaking we rather use the coordinate transformation $w^{\frac{1}{2}} = v^{\frac{1}{2}}(1+v)^{\frac{1}{2}(r^2/\rho^2-1)}$ etc.}
\end{conjecture}

This conjecture implies that the algebraicity statements of Conjectures \ref{conj:Verlindevir} and \ref{conj:Segrevir} are equivalent.
Combining Conjectures \ref{conj:SVcorrvir} and \ref{conj:SDvir}, we obtain interesting relations among the universal functions of Conjecture \ref{conj:Segrevir}.  
In the cases where we have explicit (conjectural) algebraic expressions for $A_r$, $B_r$, $A_{i,r}$, $B_{ij,r}$, $Y_{\rho+r}$, $Z_{\rho+r}$, $Y_{i,\rho+r}$, $Z_{ij,\rho+r}$, we show that they satisfy the equations of this conjecture. We give some examples of this in Section \ref{sec:alg}.

\subsection{Algebraicity} \label{sec:alg}

As mentioned in Sections \ref{sec:verlinde:rkgeneral} and \ref{sec:Segre:arbrk}, the algebraicity part of Conjectures \ref{conj:Verlindevir} and  \ref{conj:Segrevir} are supported by explicit conjectural formulae for $A_r$, $B_r$, $A_{i,r}$, $B_{ij,r}$, $Y_s$, $Z_s$, $Y_{i,s}$, $Z_{ij,s}$ for several values of $\rho,r,s$. In this section, we present three examples of such formulae. They are verified on a list of surfaces and up to certain virtual dimensions using the methods of Section \ref{sec:struc}. See \cite{GK4} for the precise list of verifications and many more examples. The formulae we present are connected by the virtual Segre-Verlinde correspondence.
This provides checks of Conjectures \ref{conj:Verlindevir}, \ref{conj:SDvir}, \ref{conj:Segrevir}, and \ref{conj:SVcorrvir}. \\

\noindent {\bf Example 1 ($\rho=2$).} For $\rho=2$ we conjecture
\begin{align*}
Y_{1,s}(z^{\frac{1}{2}})&=\frac{Y_{s}(-z^{\frac{1}{2}})}{Y_{s}(z^{\frac{1}{2}})}, \quad Z_{11,s}(z^{\frac{1}{2}})=\frac{Z_{s}(-z^{\frac{1}{2}})}{Z_{s}(z^{\frac{1}{2}})},\\
A_{1,r}(w^{\frac{1}{2}})&=\frac{A_r(-w^{\frac{1}{2}})}{A_r(w^{\frac{1}{2}})}, \, \, \, B_{11,r}(w^{\frac{1}{2}})=\frac{B_r(-w^{\frac{1}{2}})}{B_r(w^{\frac{1}{2}})},
\end{align*}
for any $s, r \in \Z$. For $s=1$, $r=-1$, $z=t(1+\smfr{1}{2}t)^{\frac{1}{2}}$, and $w=v(1+v)^{-\frac{3}{4}}$ 
we conjecture
\begin{align*}Y_1(z^{\frac{1}{2}})&=(1+t)+t^{\frac{1}{2}}(1+\smfr{3}{4}t)^{\frac{1}{2}},\quad
Z_1(z^{\frac{1}{2}})=\frac{1+\smfr{3}{4}t}{1+\smfr{1}{2}t}-\smfr{1}{2} t^{\frac{1}{2}} \frac{(1+\smfr{3}{4}t)^{\frac{1}{2}}}{1+\smfr{1}{2}t},\\
A_{-1}(w^{\frac{1}{2}})&=1+\smfr{1}{2}v+v^{\frac{1}{2}}(1+\smfr{1}{4}v)^{\frac{1}{2}}\quad
B_{-1}(w^{\frac{1}{2}})=1+\smfr{1}{4}v- \smfr{1}{2} v^{\frac{1}{2}} (1+\smfr{1}{4}v)^{\frac{1}{2}}.
\end{align*}
Taking
$v=t(1+\smfr{1}{2}t)^{-1}$, this is consistent with the virtual Segre-Verlinde correspondence. \\

\noindent {\bf Example 2 ($\rho=2$).} For $\rho=2$, $s=3$, $r=1$, $z=t(1-\smfr{1}{2}t)^{-\frac{1}{2}}$, and $w=v(1+v)^{-\frac{3}{4}}$, we conjecturally have
\begin{align*}
Y_{3}(z^{\frac{1}{2}})&=1+t^{\frac{1}{2}}(1-\smfr{1}{4}t)^{\frac{1}{2}},\\
Z_3(z^{\frac{1}{2}})&=\frac{1+\smfr{1}{2}t}{(1-\smfr{1}{2}t)^3} ((1-\smfr{1}{4}t)(1+\smfr{1}{2}t)-\smfr{3}{2}t^{\frac{1}{2}}(1-\smfr{1}{4}t)^{\frac{1}{2}}(1-\smfr{1}{6}t)),\\
A_{1}(w^{\frac{1}{2}})&=\frac{1+\smfr{1}{2}v+v^{\frac{1}{2}}(1+\smfr{1}{4}v)^{\frac{1}{2}}}{1+v},\\
B_{1}(w^{\frac{1}{2}})&=(1+v)((1+v)(1+\smfr{1}{4}v)-\smfr{3}{2}v^{\frac{1}{2}}(1+\smfr{1}{3}v)(1+\smfr{1}{4}v)^{\frac{1}{2}}).
\end{align*}
Taking
$v=t(1-\smfr{1}{2}t)^{-1}$, this is consistent with the virtual Segre-Verlinde correspondence.
Together with the previous example, we also immediately obtain the relations 
$$A_{\pm 1}(w^{\frac{1}{2}})=(1+v)^{-\frac{1}{2}} \Bigg( \frac{B_{\mp 1}(-w^{\frac{1}{2}})}{B_{\pm 1}(w^{\frac{1}{2}})} \Bigg)^{\frac{1}{2}}$$ predicted by virtual Serre duality (Conjecture \ref{conj:SDvir}). \\

\noindent {\bf Example 3 ($\rho=3$).} We take $\rho=3$ and define
\begin{align*}
a_1&:=(1+\smfr{2}{3}t)^{\frac{1}{2}}(2+\smfr{17}{6}t),\quad b_1:=\smfr{3}{2}t(1+\smfr{10}{9}t)^{\frac{1}{2}},\\ 
c_1&:=6t + \smfr{25}{2}t^2 + \smfr{20}{3}t^3,\quad \quad  d_1:=(6t+\smfr{17}{2}t^2)(1+\smfr{2}{3}t)^{\frac{1}{2}}(1+\smfr{10}{9}t)^{\frac{1}{2}},\\
a_2&:=(3+\smfr{10}{3}t)(1+\smfr{2}{3}t)^{\frac{1}{2}},\quad \,  b_2:=(1+\smfr{5}{3}t)(1+\smfr{10}{9}t)^{\frac{1}{2}},\\ 
c_2&:=6t+\smfr{35}{3}t^2+\smfr{50}{9}t^3,\quad \ \ \ d_2:=(6t+\smfr{20}{3}t^2)(1+\smfr{2}{3}t)^{\frac{1}{2}}(1+\smfr{10}{9}t)^{\frac{1}{2}}.
\end{align*}
On the Segre side, taking $s = 1$, $z=t(1+\smfr{2}{3}t)^{\frac{2}{3}}$, and suppressing the argument $z^{\frac{1}{2}}$, we conjecturally have
\begin{align*}
Y_1&=\smfr{1}{2}\big(a_1+b_1-\sqrt{c_1+d_1}\big),\quad \quad \ \, Y_1Y_{1,1}Y_{2,1}=\smfr{1}{2}\big(a_1+b_1+\sqrt{c_1+d_1}\big), \\ 
Y_1Y_{1,1}&=\smfr{1}{2}\big(a_1-b_1+\sqrt{c_1-d_1}\big),\quad \quad \ \, \quad \ \, Y_1Y_{2,1}=\smfr{1}{2}\big(a_1-b_1-\sqrt{c_1-d_1}\big),\\
Z_1&=\frac{a_2+b_2+\sqrt{c_2+d_2}}{2(1+\frac{2}{3}t)^{\frac{3}{2}}},\quad Z_1Z_{11,1}Z_{12,1}Z_{22,1}=\frac{a_2+b_2-\sqrt{c_2+d_2}}{2(1+\frac{2}{3}t)^{\frac{3}{2}}}, \\ 
Z_1Z_{11,1}&=\frac{a_2-b_2-\sqrt{c_2-d_2}}{2(1+\frac{2}{3}t)^{\frac{3}{2}}},\quad \quad \quad \quad \quad Z_1Z_{22,1}=\frac{a_2-b_2+\sqrt{c_2-d_2}}{2(1+\frac{2}{3}t)^{\frac{3}{2}}}.
\end{align*}
On the Verlinde side, we put 
\begin{align*}
\alpha_1&:=2+\smfr{3}{2}v,\quad \quad \quad \quad \, \  \ \beta_1:=\smfr{3}{2}v(1+\smfr{4}{9}v)^{\frac{1}{2}},\\ 
\gamma_1&:=6v + \smfr{9}{2}v^2 + v^3,\quad \quad \delta_1:=(6v+\smfr{9}{2}v^2)(1+\smfr{4}{9}v)^{\frac{1}{2}},\\
\alpha_2&:=3+\smfr{4}{3}v,\quad \quad \quad \quad \quad  \beta_2:=(1+v)(1+\smfr{4}{9}v)^{\frac{1}{2}},\\ 
\gamma_2&:=6v+\smfr{11}{3}v^2+\smfr{4}{9}v^3,\quad \, \delta_2:=(6v+\smfr{8}{3}v^2)(1+\smfr{4}{9}v)^{\frac{1}{2}}.
\end{align*}
Then for $r=-2$ and $w = v(1+v)^{-\frac{5}{9}}$, we conjecturally have
\begin{align*}
A_{-2}&=\smfr{1}{2}\big(\alpha_1+\beta_1-\sqrt{\gamma_1+\delta_1}\big),\quad \quad \quad \ \ A_{-2}A_{1,-2}A_{2,-2}=\smfr{1}{2}\big(\alpha_1+\beta_1+\sqrt{\gamma_1+\delta_1}\big), \\ 
A_{-2}A_{1,-2}&=\smfr{1}{2}\big(\alpha_1-\beta_1+\sqrt{\gamma_1-\delta_1}\big),\quad \quad \quad \quad \quad \ \, A_{-2}A_{2,-2}=\smfr{1}{2}\big(\alpha_1-\beta_1-\sqrt{\gamma_1-\delta_1}\big),\\
B_{-2}&=\smfr{1}{2}\big(\alpha_2+\beta_2+\sqrt{\gamma_2+\delta_2}\big),\quad B_{-2}B_{11,-2}B_{12,-2}B_{22,-2}=\smfr{1}{2}\big(\alpha_2+\beta_2-\sqrt{\gamma_2+\delta_2}\big), \\ 
B_{-2}B_{11,-2}&=\smfr{1}{2}\big(\alpha_2-\beta_2-\sqrt{\gamma_2-\delta_2}\big),\quad \quad \quad \quad \ \ \ \, B_{-2}B_{22,-2}=\smfr{1}{2}\big(\alpha_2-\beta_2+\sqrt{\gamma_2-\delta_2}\big).
\end{align*}
Taking $v = t(1+\frac{2}{3}t)^{-1}$, this is compatible with the virtual Segre-Verlinde correspondence.

\section{Universal functions} \label{sec:struc}

For each of the virtual invariants of Gieseker-Maruyama moduli spaces discussed in this survey, we can show that they are determined by a universal function in Chern numbers and Seiberg-Witten invariants.  The main ingredient for our universality results is Mochizuki's formula for descendent Donaldson invariants. 

After introducing Mochizuki's formula, we illustrate how to derive the universal function in the case of virtual Euler characteristics in the rank 2 case (Theorem \ref{thm:univ}). The strategy for the other virtual invariants of this survey is similar. We end this section by discussing how the universal functions can be applied to verifications of our conjectures in examples.

\subsection{Mochizuki's formula}

This section is devoted to a remarkable formula appearing in T.~Mochizuki's monograph \cite[Thm.~7.5.2]{Moc}. Let $(S,H)$ be a smooth polarized surface satisfying $b_1(S) = 0$. Consider the Gieseker-Maruyama moduli space $M:=M_S^H(\rho,c_1,c_2)$ for arbitrary $\rho>1$. We assume $M$ does not contain strictly Gieseker $H$-semistable sheaves. For the moment, we also assume $S \times M$ has a universal sheaf $\EE$ ---an assumption we get rid of in Remark \ref{univdrop2}.

For any $\alpha \in H^*(S,\Q)$ and $k \geq 2$, we consider the slant product
$$
\ch_k(\EE) / \mathrm{PD}(\alpha) \in H^*(M,\Q),
$$
where $\mathrm{PD}(\alpha)$ denotes the Poincar\'e dual of $\alpha$. For any polynomial expression $P(\EE)$ in slant products, we refer to the virtual intersection number
$$
\int_{[M]^{\vir}} P(\EE) \in \Q
$$
as a descendent Donaldson invariant of $S$. Similar to Donaldson-Thomas theory, the word ``descendent'' refers to the fact that we allow $k>2$. Mochizuki's formula reduces \emph{any} descendent Donaldson invariant to an expression involving Seiberg-Witten invariants and intersection numbers on products of Hilbert schemes of points. We introduce the required notation.

For any non-negative integers $\bsn = (n_1, \ldots, n_\rho)$, we define
$$
S^{[\bsn]} := S^{[n_1]} \times \cdots \times S^{[n_\rho]}.
$$
For a tautological vector bundle $L^{[n_i]}$ on $S^{[n_i]}$, we denote its pull-back to $S^{[\bsn]}$ by the same symbol. Let $\I_i$ be the universal ideal sheaf on $S \times S^{[n_i]}$, then we denote its pull-back to $S \times S^{[\bsn]}$ by the same symbol too. We denote its twist by the pull-back of a divisor class $a_i \in A^1(S)$ by $\I_i(a_i)$.

We endow $S^{[\bsn]}$ with the trivial action of $\mathbb{T} = (\C^{*})^{\rho-1}$. Let 
$$
\t_1, \ldots, \t_{\rho-1} \in X(\mathbb{T}) \cong \Z^{\rho-1}
$$
be the standard degree one characters of $\mathbb{T}$. Then any character of $\mathbb{T}$ is of the form $\prod_i \t_i^{w_i}$ for some $w_1, \dots, w_{\rho-1} \in \Z$. Any $\mathbb{T}$-equivariant coherent sheaf $\F$ on $S^{[\bsn]}$ decomposes into eigensheaves
$$
\F = \bigoplus_{\boldsymbol{w} = (w_1, \ldots, w_{\rho-1}) \in \Z^{\rho-1}} \F_{\boldsymbol{w}} \otimes \prod_i \t_i^{w_i}.
$$
We also endow $S \times S^{[\bsn]}$ with the trivial $\mathbb{T}$-action, then projection $\pi : S \times S^{[\bsn]} \rightarrow S^{[\bsn]}$ is obviously a $\mathbb{T}$-equivariant morphism. Moreover, we write
$$
H_{\mathbb{T}}^*(\mathrm{pt},\Z) = \Z[t_1^{\pm 1}, \ldots, t_{\rho-1}^{\pm 1}],
$$
where $t_i := c_1^{\mathbb{T}}(\t_i)$ denotes the $\mathbb{T}$-equivariant first Chern class. The following (rational) characters in $X(\mathbb{T}) \otimes_{\Z} \Q$ play an important role in Mochizuki's formula
\begin{align}
\begin{split} \label{defT}
\frakT_i &:= \t_i^{-1} \prod_{j<i} \t_j^{\frac{1}{\rho-j}}, \quad \forall i=1, \ldots, \rho-1, \quad \frakT_\rho := \prod_{j<\rho} \t_j^{\frac{1}{\rho-j}}, \\
T_i &:= c_1^{\mathbb{T}}(\frak{T}_i), \quad \forall i=1, \ldots, \rho.
\end{split}
\end{align}

For any Chern character $\ch \in H^*(S,\Q)$ on $S$, we 
define
$$
\chi(\ch) := \int_S \ch \cdot \td(S).
$$
For any Chern character $\ch = (\rho,c_1,\frac{1}{2} c_1^2-c_2)$, we denote the corresponding Hilbert polynomial by $h_{\ch}(t) = \chi(\ch \cdot e^{tH})$. For any divisor class $c \in A^1(S)$, we set $\chi(c) := \chi(e^c)$.

Let $P(\E)$ be any polynomial expression in slant products such that $$P(\E) = P(\E \otimes \mathcal{L})$$ for any $\mathcal{L} \in \Pic(S \times M)$. Then $P(\E)$ is independent of the choice of universal sheaf. For a $\mathbb{T}$-equivariant coherent sheaf $\F$ on $S \times S^{[\bsn]}$, we denote by $P(\F)$ the expression obtained from $P(\E)$ by replacing $S \times M$ by $S \times S^{[\bsn]}$, $\E$ by $\F$, and all Chern classes by $\mathbb{T}$-equivariant Chern classes. For any divisor classes $\bsa = (a_1, \ldots, a_{\rho})$, we define
\begin{align*}
&Q\Big( \I_1(a_1) \otimes \frak{T}_1, \ldots, \I_\rho(a_{\rho}) \otimes \frak{T}_{\rho} \Big) := \\
&\prod_{i<j} e\Big( - R\pi_{*} R\hom (\I_i(a_i) \otimes \frakT_i , \I_j(a_j) \otimes \frakT_j) - R\pi_{*} R\hom(\I_j(a_j) \otimes \frakT_j , \I_i(a_i) \otimes \frakT_i) \Big),
\end{align*}
where $e(\cdot)$ denotes $\mathbb{T}$-equivariant Euler class and $\pi : S \times S^{[\bsn]} \rightarrow S^{[\bsn]}$ is projection. Using Mochizuki's notation \cite[Sect.~7.5.2]{Moc}, for any non-negative integers $\bsn = (n_1, \ldots, n_\rho)$ and any divisor classes $\bsa = (a_1, \ldots, a_{\rho})$ on $S$, we define
\begin{align*}
&\widetilde{\Psi}(\bsa,\bsn,{\boldsymbol{t}}) := \Bigg( \prod_{i=1}^{\rho-1} t_i^{-1+\sum_{j \geq i} \chi(1,a_j,\frac{1}{2}a_j^2 - n_j)} \Bigg) \Bigg( \prod_{i<j} \frac{1}{(T_j - T_i)^{\chi(a_j)}} \Bigg) \\
&\cdot \frac{P\Big( \bigoplus_{i=1}^{\rho} \I_i(a_i) \otimes \frak{T}_i \Big)}{Q\Big( \I_1(a_1) \otimes \frak{T}_1, \ldots, \I_\rho(a_\rho) \otimes \frak{T}_\rho \Big)} \Bigg( \prod_{i=1}^{\rho-1} e(\O(a_i)^{[n_i]}) \Bigg) \Bigg( \prod_{i<j} e(\O(a_j)^{[n_j]} \otimes \frakT_j \frakT_i^{-1}) \Bigg).
\end{align*}
Finally, we define
\begin{align*}
\Psi(\bsa,\bsn) := \Res_{t_{1}} \cdots \Res_{t_{\rho-1}} \widetilde{\Psi}(\bsa,\bsn,{\boldsymbol{t}}),
\end{align*}
where $\Res_{t_i}(\cdot)$ takes the residue of $(\cdot)$ in the variable $t_i$ at zero, i.e.~the coefficient of $t_{i}^{-1}$ after expanding $(\cdot)$ as a Laurent series in $t_i$.

\begin{theorem}[Mochizuki] \label{mocthm}
Let $(S,H)$ be a smooth polarized surface such that $b_1(S) = 0$ and $p_g(S) > 0$. Consider the Gieseker-Maruyama moduli space $M:=M_S^H(\rho,c_1,c_2)$ for some $\rho>0$. Assume the following:
\begin{enumerate}
\item $M$ does not contain strictly semistable sheaves,
\item there exists a universal sheaf $\E$ on $S \times M$ ,
\item $h_{(\rho, c_1, \frac{1}{2} c_1^2-c_2)} / \rho > h_{e^{K_S}}$,
\item $\chi(\rho, c_1, \frac{1}{2} c_1^2-c_2) > (\rho-2) \chi(\O_S)$.
\end{enumerate}
Let $P(\EE)$ be any polynomial expression in slant products such that $P(\E) = P(\E \otimes \mathcal{L})$ for all $\mathcal{L} \in \Pic(S \times M)$. Then
\begin{align*}
\int_{[M]^{\vir}} P(\EE) = (-1)^{\rho-1} \rho \sum_{(a_1, \ldots, a_{\rho}) \atop (n_1, \cdots, n_\rho)} \prod_{i=1}^{\rho-1} \SW(a_i) \int_{S^{[\boldsymbol{n}]}} \Psi(\boldsymbol{a}, \boldsymbol{n}),
\end{align*}
where the sum is over all $(a_1, \ldots, a_{\rho}) \in H^2(S,\Z)^\rho$ and $(n_1, \ldots, n_\rho) \in \Z_{\geq 0}^\rho$ satisfying
\begin{align*}
c_1 &= a_1 + \cdots + a_{\rho}, \\
c_2 &= n_1 + \cdots + n_\rho + \sum_{i<j} a_i a_j, \\
h_{(1,a_i,\frac{1}{2}a_i^2 - n_i)} &< \frac{1}{\rho - i} \sum_{j>i} h_{(1,a_j,\frac{1}{2}a_j^2 - n_j)} \quad \forall i=1, \ldots, \rho-1.
\end{align*}
\end{theorem}

\begin{remark} \label{univdrop2}
Mochizuki derives his formula for the Deligne-Mumford stack $\mathcal{M}$ of oriented sheaves, i.e.~pairs $(E,\phi)$ where $[E] \in M$, $\phi : \det E \cong \O(c_1)$, and $\O(c_1)$ is a fixed line bundle with first Chern class $c_1$. Then $S \times \mathcal{M}$ always has a universal sheaf $\mathcal{E}$. When $\mathcal{M}$ does not contain strictly Gieseker semistable sheaves, this can be used to define descendent Donaldson invariants $\int_{[\mathcal{M}]^{\vir}} P(\mathcal{E})$ for any polynomial in slant products. Mochizuki's formula for $\int_{[\mathcal{M}]^{\vir}} P(\mathcal{E})$ only differs from the above formula by a factor $\rho$. 

Furthermore, there exists a degree $\frac{1}{\rho} : 1$ \'etale morphism $\mathcal{M} \rightarrow M$, which can be used to derive Mochizuki's formula for $\int_{[M]^{\vir}} P(\EE)$ as stated above (essentially by push-forward). Since we require $P(\E)$ to be invariant upon replacing $\E$ by $\E \otimes \mathcal{L}$, it follows that $P(\E)$ is defined without assuming the existence of a universal sheaf $\E$ on $S \times M$, so Condition (2) can be dropped from Theorem \ref{mocthm}. Finally, Mochizuki also extends his formula to the case $\mathcal{M}$ has strictly semistable sheaves, but we will not discuss this.
\end{remark}

\begin{remark} \label{strongmoc}
Conjecturally,  Condition (3) can be dropped from Theorem \ref{mocthm} and the sum in the formula can be replaced by the sum over \emph{all} $(a_1, \ldots, a_\rho) \in H^2(S,\Z)^{\rho}$ and $(n_1, \ldots, n_\rho) \in \Z_{\geq 0}^{\rho}$, i.e.~without imposing the inequalities (see also \cite{GNY3, GK1}). Condition (4) is essential and cannot be dropped.
\end{remark}

\subsection{Universal function}

We now derive a universal function that determines the virtual Euler characteristics of all rank 2 Gieseker-Maruyama moduli spaces on any smooth polarized surface $(S,H)$ satisfying $b_1(S) = 0$ and $p_g(S)>0$. For each of the virtual invariants in this survey, we have a similar universal function derived by a similar proof \cite{GK1, GK2, GK3, GK4, GKW}.
\begin{theorem} \cite{GK1} \label{thm:univ}
There exist  
$
A_1(t,q), \ldots, A_7(t,q) \in 1 + q \, \Q(t)[[q]]
$
with the following property. 
Let $(S,H)$ be any smooth polarized surface such that $b_1(S) = 0$ and $p_g(S) > 0$. Consider $M:=M_S^H(2,c_1,c_2)$ and assume the following:
\begin{enumerate}
\item[(a)] $M$ does not contain strictly Gieseker $H$-semistable sheaves,
\item[(b)] $h_{(2, c_1, \frac{1}{2} c_1^2-c_2)} / 2 > h_{e^{K_S}}$,
\item[(c)] $\chi(2, c_1, \frac{1}{2} c_1^2-c_2) > 0$,
\item[(d)] for any $a_1, a_2 \in H^2(S,\Z)$ such that $a_1$ is a Seiberg-Witten basic class, $a_1 + a_2=c_1$, and $a_1 H \leq a_2 H$, the inequality is strict.
\end{enumerate}
Then $e^{\vir}(M)$ equals $\Res_{t}$ of the coefficient of $x^{\vd(M)}$ of the following expression
\begin{align*}
&-2  \!\!\!\!\! \sum_{(a_1,a_2) \in H^2(S,\Z)^{2} \atop a_1+a_2 = c_1 \, \mathrm{and} \, a_1 H \leq a_2 H} \!\! \SW(a_1) \, 2^{ - \chi(a_2)} \, t^{\chi(\O_S)-1}  \Big( \frac{2t}{1+2t} \Big)^{\chi(a_2-a_1)} \Big( \frac{-2t}{1-2t} \Big)^{\chi(a_1-a_2)} x^{-(a_1-a_2)^2-3 \chi(\O_S)} \\
&\cdot A_1(t,x^4)^{a_1^2} \, A_2(t,x^4)^{a_1 a_2} \, A_3(t,x^4)^{a_2^2} \, A_4(t,x^4)^{a_1 K_S} \, A_5(t,x^4)^{a_2 K_S} \, A_6(t,x^4)^{K_S^2} \, A_7(t,x^4)^{\chi(\O_S)}.
\end{align*}
\end{theorem}

\begin{proof}
\noindent \textbf{Reduction to Donaldson invariants.} We first express virtual Euler characteristics in terms of descendent Donaldson invariants. We denote projections to the factors of $S \times M$ by $\pi_S$ and $\pi_M$ respectively. Recall the virtual Poinca\'re-Hopf formula 
$$
e^{\vir}(M) = \int_{[M]^{\vir}} c(T_M^{\vir}),
$$
where $c$ is total Chern class and $T_M^{\vir} = R\pi_{M*} R\hom(\E,\E)_0[1]$ (see \eqref{Tvir}). By Grothendieck-Riemann-Roch, we can express $c(T_M^{\vir})$ as a polynomial in expressions of the following form
$$
\pi_{M*} \big( \pi_S^* \alpha \cdot \ch_a(\EE) \cdot \ch_b(\EE) \big), 
$$
where $\alpha$ is a component of $\td(S)$. Next, we write each such expression as a polynomial in slant products. Denote by $\pi_{ij}$ and $\pi_i$ the projections from $S \times S \times M$ to factors $(i,j)$ and $i$ respectively. Then 
\begin{equation} \label{MxSxS}
\pi_{M*} \big( \pi_S^* \alpha \cdot \ch_a(\EE) \cdot \ch_b(\EE) \big)  = \pi_{3*} \big( \pi_1^* \alpha \cdot \pi_{12}^{*} \Delta \cdot \pi_{23}^{*} \ch_a(\E) \cdot \pi_{13}^{*} \ch_b(\E)\big),
\end{equation}
where $\Delta \in H^4(S \times S,\Q)$ is (Poincar\'e dual to) the class of the diagonal. Consider the K\"unneth decomposition 
$$
\Delta = \sum_{i+j=4} \theta_{1}^{(i)} \boxtimes \theta_{2}^{(j)},
$$
where $\theta_{1}^{(i)} \in H^i(S,\Q)$ and $\theta_{2}^{(j)} \in H^j(S,\Q)$. Substituting into \eqref{MxSxS} and using the projection formula yields
\begin{equation*} 
\pi_{M*} \big( \pi_{S}^{*} \alpha \cdot \ch_a(\E) \cdot \ch_b(\E) \big) = \sum_{i+j=4} ( \ch_a(\E) / \alpha \theta_{1}^{(i)}) \cdot ( \ch_b(\E) / \theta_{2}^{(j)}). 
\end{equation*}

\noindent \textbf{Leading term.} By Theorem \ref{mocthm}, we can express $e^{\vir}(M)$ as $\mathrm{Res}_t$ of
$$
-2 \sum_{(a_1,a_2) \in H^2(S,\Z)^{2} \atop a_1+a_2 = c_1 \, \textrm{and} \, a_1 H \leq a_2 H}  \sum_{n_1 + n_2 =c_2 - a_1a_2} \SW(a_1) \int_{S^{[n_1]} \times S^{[n_2]}} \widetilde{\Psi}(a_1,a_2,n_1,n_2,t).
$$
We isolate the part involving intersection numbers on Hilbert schemes and put them into a separate generating function as follows
$$
\sum_{n_1,n_2 \geq 0} q^{n_1+n_2} \int_{S^{[n_1]} \times S^{[n_2]}} \widetilde{\Psi}(a_1,a_2,n_1,n_2,t).
$$
We define its constant term by 
$$
\mathsf{C}(a_1,a_2,t) := \widetilde{\Psi}(a_1,a_2,0,0,t),
$$ 
i.e.~the term corresponding to $n_1=n_2=0$. Defining $t:=t_1$, we obtain 
\begin{align*}
&\mathsf{C}(a_1,a_2,t) = t^{-1 + 2 \chi(\O_S)+ \frac{1}{2} a_1(a_1-K_S) + \frac{1}{2} a_2(a_2-K_S) }  (T_2 - T_1)^{\chi(a_2-a_1) - \chi(a_2)} (T_1-T_2)^{\chi(a_1-a_2)}  \\
&\cdot c^{\mathbb{T}}(R\Gamma(S,\O_S)-R\Hom_S(\O_S(a_1) \otimes \mathfrak{T}_1 \oplus \O_S(a_2) \otimes \mathfrak{T}_2, \O_S(a_1) \otimes \mathfrak{T}_1 \oplus \O_S(a_2) \otimes \mathfrak{T}_2)).
\end{align*}
By \eqref{defT}, we have $T_1 = -t$ and $T_2=t$, hence
\begin{align*}
\mathsf{C}(a_1,a_2,t) = t^{-1 + 2 \chi(\O_S)+ \frac{1}{2} a_1(a_1-K_S) + \frac{1}{2} a_2(a_2-K_S) }  (2t)^{ - \chi(a_2)} \Big( \frac{2t}{1+2t} \Big)^{\chi(a_2-a_1)} \Big( \frac{-2t}{1-2t} \Big)^{\chi(a_1-a_2)}.
\end{align*}
Furthermore, when $a_1$ is a Seiberg-Witten basic class, we have $a_1^2 = a_1 K_S$. \\

\noindent \textbf{Multiplicativity.} Let $S$ be any \emph{possibly disconnected} smooth projective surface and let $a_1,a_2 \in A^1(S)$ be \emph{arbitrary} divisor classes on $S$. Define the generating function
\begin{equation}\label{defZ}
\mathsf{Z}_S(a_1,a_2,t,q) := \frac{1}{\mathsf{C}(a_1,a_2,t)} \sum_{n_1,n_2 \geq 0} q^{n_1+n_2} \int_{S^{[n_1]} \times S^{[n_2]}} \widetilde{\Psi}(a_1,a_2,n_1,n_2,t).
\end{equation}
We claim that for any $(S',a_1',a_2')$ and $(S'',a_1'',a_2'')$, we have 
\begin{align}
\begin{split} \label{mult}
&\mathsf{Z}_{S' \sqcup S''}(a_1' \sqcup a_1'', a_2' \sqcup a_2'', t, q) =  \mathsf{Z}_{S'}(a_1',a_2',t,q) \, \mathsf{Z}_{S''}(a_1'',a_2'',t,q).
\end{split}
\end{align}
This follows from the decompositions 
\begin{align*}
&(S' \sqcup S'')^{[n_1]} \times (S' \sqcup S'')^{[n_2]} = \bigsqcup_{n_{11}+n_{12}=n_1} \bigsqcup_{n_{21}+n_{22} = n_2} S^{\prime [n_{11}]} \times S^{\prime [n_{21}]} \times S^{\prime\prime [n_{12}]} \times S^{\prime\prime [n_{22}]}, \\
&\bigoplus_{i=1}^{2} \I_i(a_i) \otimes \frakT_i  \Big|_{S^{\prime[n_{11}]} \times S^{\prime[n_{21}]} \times S^{\prime\prime [n_{12}]} \times S^{\prime\prime [n_{22}]}} = \bigoplus_{i=1}^{2}  \I'_i(a'_i) \otimes \frakT_i \oplus  \bigoplus_{i=1}^{2} \I''_i(a''_i) \otimes \frakT_i,
\end{align*}
where we suppress various pull-backs, combined with the identity $c(V + W) = c(V) c(W)$ for the total Chern class. \\

\noindent \textbf{Universality.} By the general universality property \cite[Thm.~4.1]{EGL}, there exists a universal function\footnote{More precisely, \cite[Thm.~4.1]{EGL} only deals with intersection numbers on a single Hilbert scheme. The extension to intersection numbers on products of Hilbert schemes was established in \cite[Sect.~5]{GNY1}.}
$$
G(x_1, \ldots, x_7,t,q) \in \Q[x_1,\ldots, x_7](t)[[q]],
$$
such that for any $(S,a_1,a_2)$ we have
\begin{equation} \label{expG}
\mathsf{Z}_S(a_1,a_2,t,q) = \exp G(a_1^2,a_1a_2,a_2^2,a_1K_S,a_2K_S,K_S^2,\chi(\O_S),t,q).
\end{equation}
Exponentiation is possible because $\mathsf{Z}_S(a_1,a_2,t,q)$ starts with a 1 (due to normalization by $\mathsf{C}(a_1,a_2,t)$). 

We now combine \eqref{mult} and \eqref{expG} in order to construct the universal functions $A_i(t,q)$. This follows from a (by now) standard cobordism argument used in several different settings in modern enumerative geometry (notably \cite{Got3, GNY1}). More precisely, we choose seven triples $(S^{(i)},a_1^{(i)},a_2^{(i)})$ such that the vectors 
$$
w_i := ((a_1^{(i)})^2,a_1^{(i)}a_2^{(i)},(a_2^{(i)})^2,a_1^{(i)}K_{S^{(i)}},a_2^{(i)}K_{S^{(i)}},K_{S^{(i)}}^2,\chi(\O_{S^{(i)}})) \in \Q^{7}
$$
form a $\Q$-basis. Now consider an arbitrary triple $(S,a_1,a_2)$. Then we can decompose $w = (a_1^2, \ldots, \chi(\O_S))$ as
$w = \sum_i n_i w_i$ for some $n_i \in \Q$.
If all $n_i \in \Z_{\geq 0}$, then \eqref{mult} implies 
\begin{equation} \label{intermed}
\sfZ_S(a_1,a_2,t,q) = \prod_{i=1}^{7} \big(\exp G(w_i,t,q) \big)^{n_i} = \exp \Big( \sum_{i=1}^{7} n_i G(w_i,t,q) \Big).
\end{equation}
Denote by $W$ the matrix with column vectors $w_1, \ldots, w_{7}$ and let $M = (m_{ij})$ be its inverse. We define 
$$
A_j(t,q) := \exp\Big(\sum_i m_{ij} G(w_i,t,q)\Big), \quad \forall j=1, \ldots, 7.
$$ 
Then \eqref{intermed} finally yields
\begin{align} 
\begin{split} \label{ZA}
&\sfZ_S(a_1,a_2,t,q) = \\
&A_1(t,q)^{a_1^2} \, A_2(t,q)^{a_1 a_2} \, A_3(t,q)^{a_2^2} \, A_4(t,q)^{a_1 K_S} \, A_5(t,q)^{a_2 K_S} \, A_6(t,q)^{K_S^2} \, A_7(t,q)^{\chi(\O_S)}.
\end{split}
\end{align}
Since the points $w = \sum_i n_i w_i$, with $n_i \in \Z_{\geq 0}$, lie Zariski dense in $\Q^{7}$, we deduce that \eqref{ZA} holds for \emph{all} triples $(S,a_1,a_2)$.
\end{proof}

\begin{remark}
Consider Conditions (a)--(d) of Theorem \ref{thm:univ}. By Remark \ref{strongmoc}, Conditions (b) and (d) can be conjecturally dropped and the sum over ``$a_1+a_2=c_1$ satisfying $a_1H \leq a_2H$'' can be replaced by the sum over all ``$a_1+a_2=c_1$''. Some of the verifications of the conjectures mentioned in this survey are unconditional, whereas others assume that Conditions (b) and (d) can be dropped and we can sum over all ``$a_1+a_2 = c_1$''. See \cite[Sect.~7]{GK1} for details.

Condition (c) in Theorem \ref{thm:univ} is necessary (Remark \ref{strongmoc}). Condition (c) gives an upper bound on $c_2$. However, this upper bound can be made arbitrarily large as follows. The map $- \otimes \O_S(mH)$ induces an isomorphism on Gieseker-Maruyama moduli spaces and it does not change our virtual invariants. However, the upper bound on $c_2$ coming from Condition (c) becomes arbitrarily large for $m \rightarrow \infty$. Therefore, in principle, we can apply Theorem \ref{thm:univ} for arbitrarly large values of $c_2$. See also \cite[Sect.~6.1]{GK1}.
\end{remark}

\begin{remark}
For each of the virtual invariants discussed in this survey, we have a universal function similar to the one in Theorem \ref{thm:univ}. The number of universal functions $A_i$, the expression for the leading term $\mathsf{C}$, and the expression for $\sfZ_{S}$ are of course different in each situation. Nonetheless, the strategy is always the same as in the proof of Theorem \ref{thm:univ}. In particular, in the first step we reduce the virtual invariant to an expression in terms of descendent Donaldson invariants. In the case of virtual $\chi_y$-genera, elliptic genera, and Verlinde numbers this requires the virtual Hirzebruch-Riemann-Roch theorem \cite{CFK, FG}. In the case of virtual cobordism classes, this requires a theorem of Shen \cite{She} stating that $\pi_*[M]_{\Omega}^{\vir}$ can be expressed in terms of the collection of virtual Chern numbers of $M$, cf.~\eqref{vircobChern}.
\end{remark}

\subsection{Toric calculations} \label{sec:toric}

The proof of Theorem \ref{thm:univ} expresses the universal functions $A_i(t,q)$ explicitly in terms of intersection numbers on Hilbert schemes of points. We now show how this provides an algorithm for calculating $A_i(t,q)$ up to, in principle, any order in $q$. Once we know all universal functions $A_i(t,q)$ explicitly up to a certain order in $q$, we can apply Theorem \ref{thm:univ} to perform the verifications mentioned in Section \ref{sec:eulervir:rk2}. The same strategy was used for the verifications of the other virtual invariants in this survey.

Recall that for \emph{any} possibly disconnected smooth projective surface $S$ (not necessarily satisfying $b_1(S)=0$ or $p_g(S)>0$!) and any $a_1, a_2 \in A^1(S)$, we defined $\sfZ_S(a_1,a_2,t,q)$ by equation \eqref{defZ}. Furthermore, we showed that there exist $A_1, \ldots, A_7 \in 1 + q \, \Q(t)[[q]]$ such that for any $(S,a_1,a_2)$ we have (cf.~\eqref{ZA})
 \begin{align*} 
\sfZ_S(a_1,a_2,t,q) = A_1^{a_1^2} \, A_2^{a_1 a_2} \, A_3^{a_2^2} \, A_4^{a_1 K_S} \, A_5^{a_2 K_S} \, A_6^{K_S^2} \, A_7^{\chi(\O_S)}.
\end{align*}
Consider the following triples
 \begin{align*}
(S,a_1,a_2) = \, &(\PP^2,0,0), (\PP^2,H,0), (\PP^2,0,H), (\PP^2,H,H), \\
&(\PP^1 \times \PP^1,0,0), (\PP^1 \times \PP^1,H_1,0), (\PP^1 \times \PP^1,0,H_1),
\end{align*}
where $H \subset \PP^2$ is the class of a line and $H_1:=\{\pt\} \times \PP^1$. Then the corresponding vectors of Chern numbers form a basis of $\Q^7$ and the universal functions $A_i(t,q)$ are determined by the generating functions $\sfZ_S(a_1,a_2,t,q)$ for the above seven triples.

Note that $S = \PP^2$ and $S = \PP^1 \times \PP^1$ are toric surfaces with dense open torus $T = (\C^*)^2$. Moreover, the chosen divisors $a_1,a_2$ are $T$-invariant. The action of $T$ on $S$ lifts to an action of $T$ on $S^{[n]}$ for each $n$. Therefore, we can apply the Atiyah-Bott localization formula to the coefficients of the generating function $\sfZ_S(a_1,a_2,t,q)$. 

The calculation of intersection numbers on Hilbert schemes of points on toric surfaces is a well-studied subject, e.g.~\cite{ES} is one of the classical references. The fixed locus $(S^{[n]})^T$ consists of isolated reduced points. More precisely, we can cover $S$ by maximal $T$-invariant affine open subsets 
$$
\{U_{\sigma} \cong \Spec \C[x_{\sigma}, y_{\sigma} ]\}_{\sigma=1}^{e(S)}
$$
and the fixed locus $(S^{[n]})^T$ precisely consists of all collections of monomial ideals
$$
\{ I_\sigma \subset \C[x_{\sigma}, y_{\sigma}] \}_{\sigma=1}^{e(S)}
$$ 
of total colength $n$. In turn, monomial ideals of finite colength in $\C[x, y]$ are in bijective correspondence with partitions. Specifically, $\lambda = (\lambda_1 \geq  \cdots \geq  \lambda_\ell)$ corresponds to the ideal
$$
\big(y^{\lambda_1}, x y^{\lambda_2}, \ldots, x^{\ell-1}y^{\lambda_\ell}, x^{\ell}\big),
$$
where $\ell(\lambda) = \ell$ denotes the length of $\lambda$. Hence we can index the fixed locus $(S^{[n]})^T$ by collections of partitions 
$$
{\boldsymbol{\lambda}} = \{ \lambda^{(\sigma)} \}_{\sigma=1}^{e(S)}
$$
of total size 
$$
\sum_{\sigma=1}^{e(S)} |\lambda^{(\sigma)}| = \sum_{\sigma=1}^{e(S)} \sum_{i=1}^{\ell(\lambda^{(\sigma)})} \lambda_i^{(\sigma)} = n.
$$
We denote the closed subscheme corresponding ${\boldsymbol{\lambda}}$ by $Z_{{\boldsymbol{\lambda}}}$. It is well-known how to determine explicit expressions for 
$$
T_{S^{[n]}} |_{Z_{{\boldsymbol{\lambda}}}}, \quad L^{[n]}|_{Z_{{\boldsymbol{\lambda}}}} = H^0(L|_{Z_{{\boldsymbol{\lambda}}}}) \in K_0^{T}(\pt) = \Z[s_1^{\pm1}, s_2^{\pm1}],
$$
where $s_1, s_2$ are the equivariant parameters of $T$. In order to calculate the $K$-group classes coming from $T_M^{\vir}$, the following lemma is useful \cite[Prop.~4.1]{GK1}. 
\begin{lemma}
Let $W$ and $Z$ be 0-dimensional $T$-invariant subschemes supported on a maximal $T$-invariant affine open subset $U_\sigma$ of a smooth projective toric surface $S$. Suppose we choose coordinates such that $U_\sigma = \Spec \C[x,y]$ and the torus action is given by $(s_1,s_2) \cdot (x,y) = (s_1x,s_2y)$. Let $D$ be a $T$-invariant divisor on $S$ and denote the character corresponding to $D|_{U_\sigma}$ by $\chi(s_1,s_2)$. Then
$$
R\Hom_S(\O_W, \O_Z(D)) = \chi(s_1,s_2) \, W^* Z \frac{(1-s_1) (1- s_2)}{s_1s_2} \in K_0^T(\pt),
$$
where $W^*$ and $Z$ denote the classes of the $T$-representations of $H^0(\O_W)^*$ and $H^0(\O_Z)$.
\end{lemma}

Using the method described in this section, we determined the universal functions $A_i(t,q)$ up to order $q^{30}$. For instance, the first few coefficients of $A_7(t,tq)$ are
{\scriptsize{
\begin{align*}
A_7(t,tq) =&\, 1+ \left( 24\,t-\frac{6}{t} \right) q+ \left( 360\,{t}^{2}-180+\frac{30}{t^2}-\frac{9}{4 t^4}+{\frac {3}{32\,{t}^{6}}} \right) {q}^{2} \\
&+ \left( 
4160\,{t}^{3}-3200\,t+\frac{1020}{t}-\frac{210}{t^3}+{\frac {135}{4\,{t}
^{5}}}-{\frac {55}{16\,{t}^{7}}}+{\frac {5}{32\,{t}^{9}}} \right) {q}^{3} \\
&+ \left( 40560\,{t}^{4}-43380\,{t}^{2}+20280-\frac{6480}{t^2}+{\frac 
{7065}{4\,{t}^{4}}}-{\frac {6255}{16\,{t}^{6}}}+{\frac {975}{16\,{t}^{
8}}}-{\frac {735}{128\,{t}^{10}}}+{\frac {495}{2048\,{t}^{12}}}
 \right) {q}^{4} +O(q^5).
\end{align*}
}}

Atiyah-Bott localization can also be used to express $\sfZ_S(a_1,a_2,t,q)$ in terms of the Nekrasov partition function with one fundamental matter and one adjoint matter. This is worked out in \cite[App.~B]{GK1}. This may provide a first step towards an approach to Conjecture \ref{conj:evir:rk2} along the lines of \cite{GNY3}.

{\tt{gottsche@ictp.it, m.kool1@uu.nl}}

\begin{thebibliography}{DMVV}
\bibitem[Al1]{Al1} S.~Alexandrov, \emph{Vafa-Witten invariants from modular anomaly}, Comm.~Numb.~Theor.~and Phys.~15 (2021) 149--219.
\bibitem[Al2]{Al2} S.~Alexandrov, \emph{Rank $N$ Vafa-Witten invariants, modularity and blow-up}, arXiv:2006.10074.
\bibitem[AMP]{AMP} S.~Alexandrov, J.~Manschot, and B.~Pioline, \textit{$S$-duality and refined BPS indices}, Comm.~Math.~Phys.~380 (2020) 755--810.
\bibitem[AGDP]{AGP} J.~E.~Andersen S.~Gukov, and Du Pei, \textit{The Verlinde formula for Higgs bundles}, arXiv:1608.01761.
\bibitem[Beau]{Beau} A. Beauville, \textit{Counting rational curves on K3 surfaces}, Duke Math.~J.~97 (1999) 99--108.
\bibitem[Beh]{Beh} K.~Behrend, \textit{Donaldson-Thomas type invariants via microlocal geometry}, Annals of Math.~170 (2009) 1307--1338.
\bibitem[BF]{BF} K.~Behrend and B.~Fantechi, \textit{The intrinsic normal cone}, Invent.~Math.~128 (1997) 45--88.
\bibitem[Bor]{Bor} R.~E.~Borcherds, \textit{Automorphic forms on $O_{s+2,2}(\R)$ and infinite products}, Invent.~Math.~120 (1995) 161--213.
\bibitem[BL1]{BL1} L.~A.~Borisov and A.~Libgober, \textit{Elliptic genera of singular varieties}, Duke Math.~Jour.~116 (2003) 319--351.
\bibitem[BL2]{BL2} L.~A.~Borisov and A.~Libgober, \textit{McKay correspondence for elliptic genera}, Annals of Math.~161 (2005) 1521--1569.
\bibitem[Bri]{Bri} T.~Bridgeland, \textit{Fourier-Mukai transforms for elliptic surfaces}, J.~reine angew.~Math.~498 (1998) 115--133.
\bibitem[BL]{BL} J.~Bryan and C.~Leung, \textit{The enumerative geometry of K3 surfaces and modular forms}, J.~Amer.~Math.~Soc.~13 (2000) 371--410.
\bibitem[Cal]{Cal} A.~C$\breve{\textrm{a}}$ld$\breve{\textrm{a}}$raru, \textit{Derived categories of twisted sheaves on Calabi-Yau manifolds}, PhD thesis Cornell University (2000).
\bibitem[Che]{Che} X.~Chen, \textit{Rational curves on K3 surfaces}, J.~Alg.~Geom.~8 (1999) 245--278.
\bibitem[CFK]{CFK} I.~Ciocan-Fontanine and M.~Kapranov, \textit{Virtual fundamental classes via dg-manifolds}, Geom.~Topol.~13 (2009) 1779--1804.
\bibitem[DMVV]{DMVV} R.~Dijkgraaf, G.~Moore, E.~Verlinde, and H.~Verlinde, \textit{Elliptic genera of symmetric products and second quantized strings}, Commun.~Math.~Phys.~185 (1997) 197--209.
\bibitem[DPS]{DPS} R.~Dijkgraaf, J.-S.~Park, and B.~J.~Schroers, \textit{$N=4$ supersymmetric Yang-Mills theory on a K\"ahler surface}, hep-th/9801066 ITFA-97-09.
\bibitem[EGL]{EGL} G.~Ellingsrud, L.~G\"ottsche, and M.~Lehn, \textit{On the cobordism class of the Hilbert scheme of a surface}, Jour.~Alg.~Geom.~10 (2001) 81--100. 
\bibitem[ES]{ES} G.~Ellingsrud and S.~A.~Str{\o}mme, \textit{On a cell  decomposition of the Hilbert scheme of points in the plane}, Invent.~Math.~91 (1988) 365--370.
\bibitem[FG]{FG} B.~Fantechi and L.~G\"ottsche, \textit{Riemann-Roch theorems and elliptic genus for virtually smooth schemes}, Geom.~Topol.~14 (2010) 83--115.
\bibitem[Fog]{Fog} J.~Fogarty, \textit{Algebraic families on an algebraic surface}, Am.~J.~Math.~10 (1968) 511--521.
\bibitem[GSY1]{GSY1} A.~Gholampour, A.~Sheshmani, and S.-T.~Yau, \textit{Nested Hilbert schemes on surfaces: Virtual fundamental class}, Adv.~Math.~365 (2020) 107046.
\bibitem[GSY2]{GSY2} A.~Gholampour, A.~Sheshmani, and S.-T.~Yau, \textit{Localized Donaldson-Thomas theory of surfaces}, Amer.~Jour.~Math.~142 (2020) 405--442.
\bibitem[GT1]{GT1} A.~Gholampour and R.~P.~Thomas, \textit{Degeneracy loci, virtual cycles and nested Hilbert schemes I}, Tunisian Jour.~Math.~2 (2020) 633--665.
\bibitem[GT2]{GT2} A.~Gholampour and R.~P.~Thomas, \textit{Degeneracy loci, virtual cycles and nested Hilbert schemes II}, Compos.~Math.~156 (2020) 1623--1663.
\bibitem[Gie]{Gie} D.~Gieseker, \textit{On the moduli of vector bundles on an algebraic surface}, Ann.~Math.~106 (1977) 45--60.
\bibitem[Got1]{Got1} L.~G\"ottsche, \textit{The Betti numbers of the Hilbert scheme of points on a smooth projective surface}, Math.~Ann.~286 (1990) 193--207. 
\bibitem[Got2]{Got2} L.~G\"ottsche, \textit{Change of polarization and Hodge numbers of moduli spaces of torsion free sheaves on surfaces}, Math.~Z.~223 (1996) 247--260.
\bibitem[Got3]{Got3} L.~G\"ottsche, \textit{A conjectural generating function for numbers of curves on surfaces}, Comm.~Math.~Phys.~196 (1998) 523--533.
\bibitem[Got4]{Got4} L.~G\"ottsche, \textit{Theta functions and Hodge numbers of moduli spaces of sheaves on rational surfaces}, Comm.~Math.~Phys.~206 (1999) 105--136.
\bibitem[Got5]{Got5} L.~G\"ottsche, \textit{Verlinde-type formulas for rational surfaces}, J.~Eur.~Math.~Soc.~22 (2020) 151--212.
\bibitem[Got6]{Got6} L.~G\"ottsche, \textit{Refined Verlinde formulas for Hilbert schemes of points and moduli of sheaves on K3 surfaces}, \'Epiga 4 15 (2020) 12 pages.
\bibitem[GH]{GH} L.~G\"ottsche and D.~Huybrechts, \textit{Hodge numbers of moduli spaces of stable bundles on K3 surfaces}, Int.~J.~Math.~7 (1996) 359--372.
\bibitem[GK1]{GK1} L.~G\"ottsche and M.~Kool, \textit{Virtual refinements of the Vafa-Witten formula}, Comm.~Math.~Phys.~376 (2020) 1--49.
\bibitem[GK2]{GK2} L.~G\"ottsche and M.~Kool, \textit{A rank 2 Dijkgraaf-Moore-Verlinde-Verlinde formula}, Comm.~Numb.~Theor.~and Phys.~13 (2019) 165--201.
\bibitem[GK3]{GK3} L.~G\"ottsche and M.~Kool, \textit{Refined $\mathrm{SU}(3)$ Vafa-Witten invariants and modularity}, Pure and Appl.~Math.~Quart.~14 (2018) 467--513.
\bibitem[GK4]{GK4} L.~G\"ottsche and M.~Kool, \textit{Virtual Segre and Verlinde numbers of projective surfaces}, arXiv:2007.11631.
\bibitem[GKW]{GKW}  L. G\"ottsche, M.~Kool, and R.A.~Williams, \textit{Verlinde formulae on complex surfaces: K-theoretic invariants}, Forum of Math.~Sigma Vol.~9:e5 (2021) 1--31.
\bibitem[GNY1]{GNY1} L.~G\"ottsche, H.~Nakajima, and K.~Yoshioka, \textit{Instanton counting and Donaldson invariants}, J.~Diff.~Geom.~80 (2008) 343--390.
\bibitem[GNY2]{GNY2} L.~G\"ottsche, H.~Nakajima, and K.~Yoshioka, \textit{K-theoretic Donaldson invariants via instanton counting}, Pure Appl.~Math.~Quart.~5 (2009) 1029--1111.
\bibitem[GNY3]{GNY3} L.~G\"ottsche, H.~Nakajima, and K.~Yoshioka, \textit{Donaldson = Seiberg-Witten from Mochizuki's formula and instanton counting}, Publ.~Res.~Inst.~Math.~Sci.~47 (2011) 307--359.
\bibitem[GS]{GS} L.~G\"ottsche and W.~Soergel, \textit{Perverse sheaves and the cohomology of Hilbert schemes of smooth algebraic surfaces}, Math.~Ann.~296 (1993) 235--245.
\bibitem[GY]{GY} L.~G\"ottsche and Y.~Yuan, \textit{Generating functions for $K$-theoretic Donaldson invariants and Le Potier's strange duality}, J.~Alg.~Geom.~28 (2019) 43--98. 
\bibitem[GP]{GP} T.~Graber and R.~Pandharipande, \textit{Localization of virtual classes}, Invent.~Math.~135 (1999) 487--518. 
\bibitem[GN]{GN} V.~A.~Gritsenko and V.~V.~Nikulin, \textit{Siegel automorphic form corrections of some Lorentzian Kac--Moody Lie algebras}, Amer.~J.~Math.~119 (1997) 181--224.
\bibitem[Groj]{Groj} I.~Grojnowski, \textit{Instantons and affine algebras. I: The Hilbert scheme and vertex operators}, Math.~Res.~Lett.~3 (1996) 275--291.
\bibitem[Gro]{Gro} A.~Grothendieck, \textit{Techniques de construction et th\'eor\`emes d'existence en g\'eom\'etrie alg\'ebrique. IV. Les sch\'emas de Hilbert}, S\'eminaire Bourbaki (1960/61), Vol.~6, Exp.~No.~221, 249--276.
\bibitem[H-L]{H-L} D.~Halpern-Leistner, \textit{The equivariant Verlinde formula on the moduli of Higgs bundles}, arXiv:1608.01754.
\bibitem[Hir]{Hir} F.~Hirzebruch, \textit{Elliptic genera of level $N$ for complex manifolds}, in: Differential geometric methods in theoretical physics, K.~Bleuler and M.~Werner, eds., Kluwer Acad.~Publ.~(1988) 37--63.
\bibitem[Huy]{Huy}  D.~Huybrechts,  \textit{Compact hyper-K\"ahler manifolds: basic results},  Invent.~Math.~135 (1999) 63--113.
\bibitem[HL]{HL} D.~Huybrechts and M.~Lehn, \textit{The geometry of moduli spaces of sheaves}, Cambridge University Press (2010).
\bibitem[Jia]{Jia} Y.~Jiang, \textit{Counting twisted sheaves and S-duality}, arXiv:1909.04241.
\bibitem[JK]{JK} Y.~Jiang and M.~Kool, \textit{Twisted sheaves and $\mathrm{SU}(r) / \Z_r$ Vafa-Witten theory}, Math.~Ann.~(2021) 25 pages.
\bibitem[Joh]{Joh} D.~Johnson, \textit{Universal series for Hilbert schemes and strange duality}, IMRN 2020 10 (2020) 3130--3152.
\bibitem[JOP]{JOP} D.~Johnson, D.~Oprea, and R.~Pandharipande, \textit{Rationality of descendent series for Hilbert and Quot schemes of surfaces}, Selecta Math.~27 (2021).
\bibitem[KL1]{KL1} Y.-H.~Kiem and J.~Li, \textit{Localizing virtual cycles by cosections}, JAMS 26 (2013) 1025--1050.
\bibitem[KL2]{KL2} Y.-H.~Kiem and J.~Li, \textit{Localizing virtual structure sheaves by cosections}, IMRN 2020 (2020) 8387--8417.
\bibitem[KMPS]{KMPS} A.~Klemm, D.~Maulik, R.~Pandharipande, and E.~Scheidegger, \textit{Noether-Lefschetz theory and the Yau-Zaslow conjecture}, JAMS 23 (2010) 1013--1040.
\bibitem[Kly]{Kly} A.~A.~Klyachko, \textit{Vector bundles and torsion free sheaves on the projective plane}, preprint Max Planck Institut f\"ur Mathematik (1991).
\bibitem[Koo]{Koo} M.~Kool, \textit{Euler characteristics of moduli spaces of torsion free sheaves on toric surfaces}, Geom.~Ded.~176 (2015) 241--269.
\bibitem[KST]{KST} M.~Kool, V.~Shende and R.~P.~Thomas, \textit{A short proof of the G\"ottsche conjecture}, Geom.~Topol.~15 (2011) 397--406. 
\bibitem[KT1]{KT1} M.~Kool and R.~P.~Thomas, \textit{Reduced classes and curve counting on surfaces I: theory}, Alg.~Geom.~1 (2014) 334--383. 
\bibitem[KT2]{KT2} M.~Kool and R.~P.~Thomas, \textit{Reduced classes and curve counting on surfaces II: calculations}, Alg.~Geom.~1 (2014) 384--399. 
\bibitem[Kri]{Kri} I.~M.~Krichever, \textit{Generalized elliptic genera and Baker-Akhiezer functions}, Math.~Notes 47 (1990) 132--142.
\bibitem[Laa1]{Laa1} T.~Laarakker, \textit{Monopole contributions to refined Vafa-Witten invariants}, Geom.~Topol.~24 (2020) 2781--2828.
\bibitem[Laa2]{Laa2} T.~Laarakker, \textit{Vertical Vafa-Witten invariants}, Selecta Math.~27 (2021).
\bibitem[LL]{LL} J.~M.~F.~Labastida and C.~Lozano, \textit{The Vafa-Witten theory for gauge group $\SU(N)$}, Adv.~Theor.~Math.~Phys.~5 (1999) 1201--1225.
\bibitem[Lee]{Lee} J.~Lee, \textit{Family Gromov-Witten invariants for K\"ahler surfaces}, Duke Math.~Jour.~123 (2004) 209--233.
\bibitem[Leh]{Leh} M.~Lehn, \textit{Chern classes of tautological sheaves on Hilbert schemes of points on surfaces}, Invent.~Math.~136 (1999) 157--207.
\bibitem[LM]{LM} M.~Levine and F.~Morel, \textit{Algebraic cobordism}, Springer (2007).
\bibitem[LP]{LP} M.~Levine and R.~Pandharipande, \textit{Algebraic cobordism revisited}, Invent.~Math.~176 (2009) 63--130.
\bibitem[Li]{Li} J.~Li, \textit{A note on enumerating rational curves in a K3 surface}, in ``Geometry and nonlinear partial differential equations" AMS/IP Studies in Adv.~Math.~29 (2002) 53--62.
\bibitem[LQ1]{LQ1} W.-P.~Li and Z.~Qin, \textit{On blowup formulae for the $S$-duality conjecture of Vafa and Witten}, Invent.~Math.~136 (1999) 451--482.
\bibitem[LQ2]{LQ2}  W.-P.~Li and Z.~Qin, \textit{On blowup formulae for the $S$-duality conjecture of Vafa and Witten II: the universal functions}, Math.~Res.~Lett.~5 (1998) 439--453.
\bibitem[LT]{LT} J.~Li and G.~Tian, \textit{Virtual moduli cycles and Gromov-Witten invariants of algebraic varieties}, JAMS 11 (1998) 119--174. 
\bibitem[Lim]{Lim} W.~Lim, \textit{Virtual $\chi_{-y}$-genera of Quot schemes on surfaces}, to appear in Jour.~LMS, arXiv:2003.04429.
\bibitem[LS]{LS} P.E.~Lowrey and T.~Sch\"urg, \textit{Derived algebraic cobordism}, J.~Inst.~Math.~Jussieu (2016) 15 407--443.
\bibitem[Man]{Man} J.~Manschot, \textit{The Betti numbers of the moduli space of stable sheaves of rank 3 on $\mathbb{P}^2$}, Letters in Math.~Phys.~98 (2011) 65--78.  
\bibitem[MOP1]{MOP1} A.~Marian, D.~Oprea, and R.~Pandharipande, \textit{Segre classes and Hilbert schemes of points}, Ann.~Sci.~\'Ec.~Norm.~Sup\'er.~(4) 50 (2017) 239--267.
\bibitem[MOP2]{MOP2} A.~Marian, D.~Oprea, and R.~Pandharipande, \textit{The combinatorics of Lehn's conjecture}, J.~Math.~Soc.~Japan 71 (2019) 299--308.
\bibitem[MOP3]{MOP3} A.~Marian, D.~Oprea, and R.~Pandharipande, \textit{Higher rank Segre integrals over the Hilbert scheme of points}, JEMS (2021) DOI:10.4171/JEMS/1149.
\bibitem[Mar1]{Mar1} M.~Maruyama, \textit{Moduli of stable sheaves I.}, J.~Math.~Kyoto Univ.~17 (1977) 91--126.
\bibitem[Mar2]{Mar2} M.~Maruyama, \textit{Moduli of stable sheaves II.}, J.~Math.~Kyoto Univ.~18 (1978) 557--614.
\bibitem[MPT]{MPT} D.~Maulik, R.~Pandharipande and R.~P.~Thomas, \textit{Curves on K3 surfaces and modular forms}, J.~Topol.~3 (2010) 937--996. 
\bibitem[MT]{MT} D.~Maulik and R.~P.~Thomas, in preparation.
\bibitem[Moc]{Moc} T.~Mochizuki, \textit{Donaldson type invariants for algebraic surfaces}, Lecture Notes in Math.~1972, Springer-Verlag, Berlin (2009). 
\bibitem[Mor]{Mor} J.~W.~Morgan, \textit{The Seiberg-Witten equations and applications to the topology of smooth four-manifolds}, Math.~Notes 44, Princeton Univ.~Press (1996).
\bibitem[Moz]{Moz} S.~Mozgovoy, \textit{Invariants of moduli spaces of stable sheaves on ruled surfaces}, arXiv:1302.4134.
\bibitem[Nak1]{Nak1} H.~Nakajima, \textit{Heisenberg algebra and Hilbert schemes of points on projective surfaces}, Ann.~Math.~145 (1997) 379--388.
\bibitem[Nak2]{Nak2} H.~Nakajima, \textit{Instantons on ALE spaces, quiver varieties, and Kac-Moody algebras}, Duke Math.~J.~76 (1994) 365--416.
\bibitem[Nak3]{Nak3} H.~Nakajima, \textit{Gauge theory on resolutions of simple singularities and simple Lie algebras}, IMRN (1994) 61--74.
\bibitem[NO]{NO} N. Nekrasov and A. Okounkov, \textit{Membranes and sheaves}, Alg.~Geom.~3 (2016) 320--369.
\bibitem[OG]{OG} K.~O' Grady, \textit{The weight-two Hodge structure of moduli space of sheaves on a K3 surface}, J.~Algebraic Geom.~6 (1999) 599--644.
\bibitem[OP]{OP} D.~Oprea and R.~Pandharipande, \textit{Quot schemes of curves and surfaces: virtual classes, integrals, Euler characteristics}, to appear in Geom.~Topol., arXiv:1903.08787.
\bibitem[PT]{PT} R.~Pandharipande and R.~P.~Thomas, \textit{The Katz-Klemm-Vafa conjecture for K3 surfaces}, Forum of Math.~Pi 4 (2016) 1--111.
\bibitem[Ren]{Ren} J.~V.~Rennemo, \textit{Universal polynomials for tautological integrals on Hilbert schemes}, Geom.~Topol.~21 (2017) 253--314.
\bibitem[She]{She} J.~Shen, \textit{Cobordism invariants of the moduli space of stable pairs}, J.~London Math.~Soc.~94 (2016) 427--446.
\bibitem[Sim]{Sim} C.~Simpson, \textit{Moduli of representations of the fundamental group of a smooth projective variety I}, Publ.~Math.~IHES 79 (1994) 47--129.
\bibitem[TT1]{TT1} Y.~Tanaka and R.~P.~Thomas, \textit{Vafa-Witten invariants for projective surfaces I: stable case}, Jour.~Alg.~Geom.~ 29 (2020) 603-668. 
\bibitem[TT2]{TT2} Y.~Tanaka and R.~P.~Thomas, \textit{Vafa-Witten invariants for projective surfaces II: semistable case}, Pure Appl.~Math.~Quart.~13 (2017) 517--562.
\bibitem[Tho]{Tho} R.~P.~Thomas, \emph{Equivariant K-theory and refined Vafa-Witten invariants}, Comm.~Math.~Phys.~378 (2020) 1451--1500.
\bibitem[Tod]{Tod} Y.~Toda, \textit{Stable pairs on local K3 surfaces}, J.~Diff.~Geom.~92 (2012) 285--370.
\bibitem[Tze]{Tze} Y.-j.~Tzeng, \textit{A proof of the G\"ottsche-Yau-Zaslow formula}, J.~Diff.Geom.~90 (2012) 439--472. 
\bibitem[VW]{VW} C.~Vafa and E.~Witten, \textit{A strong coupling test of $S$-duality}, Nucl.~Phys.~B 431 (1994) 3--77.
\bibitem[Ver]{Ver} E.~Verlinde, \textit{Fusion rules and modular transformations in $2d$ conformal field theory}, Nucl.~Phys.~B 300 (1988) 360--376.
\bibitem[Voi]{Voi} C.~Voisin, \textit{Segre classes of tautological bundles on Hilbert schemes of surfaces}, Alg.~Geom.~6 (2019) 186--195.
\bibitem[Wei]{Wei} T.~Weist, \textit{Torus fixed points of moduli spaces of stable bundles of rank three}, J.~Pure Appl.~Algebra 215 (2011) 2406--2422.
\bibitem[Wit]{Wit} E.~Witten, \textit{The index of the Dirac operator in loop space}, in: Elliptic curves and modular forms in algebraic geometry, P.~Landweber, ed., Springer-Verlag (1988) 161--181.
\bibitem[YZ]{YZ} S.-T.~Yau and E.~Zaslow, \textit{BPS states, string duality, and nodal curves on K3}, Nuclear Physics B 471 (1996) 503--512.
\bibitem[Yos1]{Yos1} K.~Yoshioka, \textit{The Betti numbers of the moduli space of stable sheaves of rank 2 on $\PP^2$}, J.~reine angew.~Math.~453 (1994) 193--220.
\bibitem[Yos2]{Yos2} K.~Yoshioka, \textit{The Betti numbers of the moduli space of stable sheaves of rank 2 on a ruled surface}, Math.~Ann.~302 (1995) 519--540.
\bibitem[Yos3]{Yos3} K.~Yoshioka, \textit{Number of $\FF_q$-rational points of the moduli of stable sheaves on elliptic surfaces}, Moduli of vector bundles, editor: M.~Maruyama, Lect.~Notes in Pure and Appl.~Math.~179, Marcel Dekker, New York (1996).
\bibitem[Yos4]{Yos4} K.~Yoshioka, \textit{Some examples of Mukai's reflections on K3 surfaces},  J.~reine~angew.~Math.~515 (1999) 97--123. 
\bibitem[Yos5]{Yos5} K.~Yoshioka, \textit{Moduli spaces of twisted sheaves on a projective variety}, in: Moduli spaces and arithmetic geometry (Kyoto, 2004), Adv.~Stud.~Pure Math.~45 (2006) 1--42.
\bibitem[Zag]{Zag} D.~Zagier, \textit{Nombres de classes et formes modulaires de poids 3/2}, C.~R.~Acad.~Sci.~Paris 281A (1975) 883--886.
\end{thebibliography}
\end{document}